\documentclass[twoside,11pt]{amsart}
\usepackage[all]{xy}
\CompileMatrices
\usepackage{amsfonts}
\usepackage{amssymb}
\usepackage{amsthm}
\usepackage{amsmath}
\usepackage{ifthen}
\usepackage{enumitem} 
\usepackage{verbatim} 
\usepackage{url}
\usepackage[usenames]{color}
\usepackage[draft]{todonotes}   

\allowdisplaybreaks 

 \newtheorem{thm}{Theorem}[section]
 \newtheorem{prop}[thm]{Proposition}
 
 \newtheorem{lem}[thm]{Lemma}

\theoremstyle{definition}
\newtheorem{defn}[thm]{Definition}

\theoremstyle{remark}
\newtheorem{rem}[thm]{Remark}

\renewcommand{\Im}{\mathrm{Im}}

\newcommand{\Sp}{\mathrm{Sp}}
\newcommand{\GL}{\mathrm{GL}}

\renewcommand{\Re}{\mathrm{Re}}
\renewcommand{\Im}{\mathrm{Im}}

\newcommand{\diag}{\mathrm{diag}}
\newcommand{\sgn}{\mathrm{sgn}}
\newcommand{\tr}{\mathrm{tr}\,}
\newcommand{\f}[1]{\mathfrak{{#1}}}
\newcommand{\pr}{\operatorname{pr}}
\newcommand{\ord}{\operatorname{ord}}

\newcommand{\Vol}{\operatorname{Vol}}
\newcommand{\norm}[1]{\mathop{\left\{\!\!\left\{#1\right\}\!\!\right\}}}

\newcommand{\back}{\backslash}
\newcommand{\action}[1]{\,{}^{\alpha}\! {{#1}}} 
\newcommand{\actioninv}[1]{\,{}^{\alpha^{-1}}\! {{#1}}} 

\newcommand{\transpose}[1]{\text{$^t\!#1$}} 
 
\def\sectionnam{\@empty}
\def\subsectionnam{\@empty}

\begin{document}
\title[On the Rankin-Selberg method for Siegel modular forms] 
{On the Rankin-Selberg method for vector valued Siegel modular forms}

\author{Thanasis Bouganis and Salvatore Mercuri}
\address{Department of Mathematical Sciences\\ Durham University\\
 Durham, UK.}
\email{athanasios.bouganis@durham.ac.uk\\ salvatore.mercuri@durham.ac.uk}

\maketitle In this work we use the Rankin-Selberg method to obtain results on the analytic properties of the standard $L$-function attached to vector valued Siegel modular forms. In particular we provide a detailed description of its possible poles and obtain a non-vanishing result of the twisted $L$-function beyond the usual range of absolute convergence. We remark that these results were known in this generality only in the case of scalar weight Siegel modular forms. As an interesting by-product of our work we establish the cuspidality of some theta series.


\section{Introduction}

The standard $L$ function attached to a scalar weight Siegel eigenform has been extensively studied in the literature. Its analytic properties have been investigated, among others, by Andrianov and Kalinin \cite{AK}, B\"{o}cherer \cite{B}, Shimura \cite{Sh94, Sh95,Sh00}, and Piatetski-Shapiro and Rallis \cite{PR1,PR2}. In all these works the properties of the $L$-function are read off by properties of Siegel-type Eisenstein series, which themselves are well-understood. However there are two different ways to obtain an integral expression of the $L$-function involving Siegel-type Eisenstein series. The first is what in this paper we will be calling the Rankin-Selberg method (involving a theta series and a Siegel-type Eisenstein series of the same degree) and the second is usually called the doubling method (involving the restriction of a higher degree Siegel-type Eisenstein series). 

It is now well-understood, especially thanks to the work of Shimura (see for example the discussion in \cite[Remark 6.3, (III)]{Sh95}) and \cite[proof of Theorem 28.8]{Sh00}, that the two methods are not equivalent, and that both deserve to be explored for their own merit. Especially when it comes to non-vanishing results beyond the usual range of absolute convergence, it seems that the use of the Rankin-Selberg method can be used to obtain better results than the ones obtained by employing the doubling method. Indeed such a non-vanishing theorem has been established for the scalar weight case in \cite{Sh00}, where the usual bound of $\Re(s) > 2n+1$ in \cite[Lemma 2012]{Sh00} is extended to $\Re(s) > (3n/2) + 1$ in \cite[Theorem 20.13]{Sh00}, where $n$ is the degree of the symplectic group. Furthermore, in the same book Shimura uses this result to establish algebraicity results for Siegel modular forms. The extension of the non-vanishing range has a direct consequence on the weights of the Siegel modular forms for which one can obtain algebraicity results (see \cite[Theorem 28.5 and Theorem 28.8]{Sh00}

Of course one can consider the standard $L$-function attached to non-scalar weight (vector valued) Siegel modular forms. The known results in this situation are not as general and precise as in the scalar weight situation. Actually the most general results are those of Piatetski-Shapiro and Rallis using the automorphic representation language \cite{PR1,PR2}. In \cite{PR1} they consider the doubling method and in \cite{PR2} the Rankin-Selberg method. In both papers the $L$-function is untwisted, Euler factors are removed, and the gamma factors are not given explicitly. 

The main aim of this paper is to consider the Rankin-Selberg expression in the vector valued case, and obtain precise results regarding the location and orders of poles. The main ingredient for our approach is the construction of a particular vector valued theta series, which in turn relies on the existence of some pluriharmonic polynomials studied by Kashiwara and Vergne \cite{KV}. We should remark right away that this puts some limitations on the representations we can consider as weights of the underlying Siegel modular form. Furthermore we will use this expression to obtain some non-vanishing results of the standard $L$-function similar to the one obtained by Shimura \cite{Sh00} in the scalar weight situation.  A crucial ingredient for this is the cuspidality of certain vector valued theta series for which, since the approach used in \cite{Sh00} is not applicable, we need a completely new idea.  We believe that our results have direct consequences for the algebraicity of the special $L$-values, much in the same way the results of Shimura \cite{Sh00} had in the scalar weight situation. We hope to explore these application in the near future.

\section{Vector Valued Siegel Modular Forms}

Alongside that of standard scene-setting, the intention of this section is to introduce adelic vector valued Siegel modular forms and establish their Fourier expansions. We mainly follow \cite{Sh00}.

Throughout the paper $1\leq n\in\mathbb{Z}$; $\mathbb{T}\subseteq\mathbb{C}$ is the unit circle; and we define three characters, with images in $\mathbb{T}$, on $\mathbb{C}, \mathbb{Q}_p, \mathbb{A}_{\mathbb{Q}}$ respectively by
\begin{align*}
	e(z) :&= e^{2\pi iz} \\
	e_p(z) :&= e(-\{x\}) \\
	e_{\mathbb{A}}(x) :&= e(x_{\infty})\prod_{p\in\mathbf{h}} e_p(x_p)
\end{align*}
where $\{x\}$ denotes the fractional part of $x\in\mathbb{Q}_p$, $\infty$ denotes the Archimedean place of $\mathbb{Q}$, and $\mathbf{h}$ denotes the non-Archimedean places. When convenient, for $x\in \mathbb{A}_{\mathbb{Q}}$ and a square matrix $M$, we shall also write $e_{\mathbf{h}}(x) = e_{\mathbb{A}}(x_{\mathbf{h}})$; $e_{\infty}(x) = e_{\mathbb{A}}(x_{\infty})$; $|M| = \det(M)$, $\|M\| = |\det(M)|$; $M>0$ ($M\geq 0$) to mean that $M$ is positive definite (respectively positive semi-definite); $\sqrt{M}$ to be a matrix such that $(\sqrt{M})^2 = M$; and
\[
	\text{diag}[M_1, \dots, M_{\ell}] = \begin{pmatrix} M_1 & 0 & \cdots &0 \\ 0 & M_2 & \cdots & 0 \\ \vdots & \vdots & \ddots & \vdots \\ 0 & 0 & \cdots & M_{\ell}\end{pmatrix}
\]
for square matrices $M_i$.


If $\alpha\in \GL_{2n}(\mathbb{Q})$ then put
\[
	\alpha = \begin{pmatrix} a_{\alpha} & b_{\alpha} \\ c_{\alpha} & d_{\alpha}\end{pmatrix}
\]
for $a_{\alpha}, b_{\alpha}, c_{\alpha}, d_{\alpha}\in M_n(\mathbb{Q})$. With 
\begin{align*}
	G :&= \Sp_n(\mathbb{Q}) = \{\alpha\in\GL_{2n}(\mathbb{Q})\mid \transpose{\alpha}\eta_n\alpha = \eta_n\} \hspace{20pt} \eta_n:=\begin{pmatrix} 0 & -I_n \\ I_n & 0\end{pmatrix} \\
	P:&= \{\alpha\in G\mid c_{\alpha} = 0\} \\
	\mathbb{H}_n:&=\{z = x+iy\in M_n(\mathbb{C})\mid \transpose{z} = z, y>0\},
\end{align*}
let $G_{\mathbb{A}}$ and $P_{\mathbb{A}}$ denote the adelizations of $G$ and $P$ respectively; there are the usual respective actions of $\Sp_n(\mathbb{R})$ and $G_{\mathbb{A}}$ on $\mathbb{H}_n$ given by
\begin{align*}
	\alpha\cdot z &= (a_{\alpha}z+b_{\alpha})(c_{\alpha}z+d_{\alpha})^{-1} \\
	x\cdot z &= x_{\infty}\cdot z 
\end{align*}
for $\alpha\in\Sp_n(\mathbb{R}), x\in G_{\mathbb{A}}$, and $z\in\mathbb{H}_n$; and we also have factors of automorphy
\begin{align*}
	\mu(\alpha, z) &= c_{\alpha}z+d_{\alpha} \\
	\mu(x, z) &= \mu(x_{\infty}, z).
\end{align*}

Let $V$ be a finite-dimensional complex vector space and let
\begin{align}
	\rho:\GL_n(\mathbb{C})\to GL(V)
\end{align}
be a rational representation. For any $f:\mathbb{H}_n\to V$ and $\alpha\in G_{\mathbb{A}}$ define a new function $f|_{\rho}\alpha:\mathbb{H}_n\to V$ by
\[	
	(f|_{\rho}\alpha)(z) := \rho(\mu(\alpha, z))^{-1}f(\alpha z)
\]
and it is clear that $f|_{\rho}(\alpha\beta) = (f|_{\rho}\alpha)|_{\rho}\beta$ for any two $\alpha, \beta\in G_{\mathbb{A}}$.

\begin{defn}\label{modularform} Given a congruence subgroup $\Gamma\leq G$ and $\rho$ as in (1) then $\mathcal{M}_{\rho}(\Gamma)$ denotes the complex vector space of all holomorphic $f:\mathbb{H}_n\to V$ such that
	\begin{enumerate}
		\item $f|_{\rho}\gamma = f$ for all $\gamma\in\Gamma$;
		\item $f$ is holomorphic at all cusps.
	\end{enumerate}

\end{defn}

The last condition is needed only in the case of $n=1$. In order to explain it, and also introduce the notion of a cusp form, we define the sets of symmetric matrices
\begin{align*}
	S:&= \{\tau\in M_n(\mathbb{Q})\mid \transpose{\tau} = \tau\} \\
	S_{+} :&= \{\tau\in S\mid \tau\geq 0\} \\
	S^{+} :&= \{\tau\in S\mid \tau > 0\} \\
	S(\mathfrak{x}) :&= S\cap M_n(\mathfrak{x}) \\
	S_{\mathbf{h}}(\f{x}) :&= \prod_p S(\mathfrak{x})_p
\end{align*}
for some fractional ideal $\mathfrak{x}$.
For any $f\in\mathcal{M}_{\rho}(\Gamma)$ and any $\gamma \in G$ we have a Fourier expansion of the form
\begin{align}
	(f|_{\rho} \gamma)(z) = \sum_{\tau\in S_{+}}c(\tau)e_{\infty}(\tau z),
\end{align}
where $c(\tau)\in V$ and $c(\tau) = 0$ for all $\tau\not\in L$ for some $\mathbb{Z}$-lattice $L$. This is automatic in the case of $n > 1$ and is the condition we impose in the case of $n=1$. We
let $\mathcal{S}_{\rho}(\Gamma)$ denote the subspace of cusp forms, that is those $f \in M_{\rho}(\Gamma)$ with the property that in the expansion $(2)$ above the sum is running over $\tau \in S^+$.

To talk about adelic vector valued modular forms we focus on representations of the form $\rho_k := \det^k\otimes\rho$ where $\rho$ is as in (1), $k\in\frac{1}{2}\mathbb{Z}$, and the definition of $\det^k$ depends on whether $k$ is an integer or not. If $k\not\in\mathbb{Z}$ then we need to use the metaplectic group $Mp_n(\mathbb{Q})$, its adelization $M_{\mathbb{A}}$, and a particular subgroup $\mathfrak{M}\leq M_{\mathbb{A}}$ (whose definition can be found in [\cite{Sh00}, p.129]) to define the appropriate factors of automorphy. For any $\sigma\in\mathfrak{M}$, whose image under the natural projection map $\pr:M_{\mathbb{A}}\to G_{\mathbb{A}}$ is pr$(\sigma) = \alpha$, we put $x_{\sigma} = x_{\alpha}$ for $x\in\{a, b, c, d\}$, $\sigma\cdot z = \alpha\cdot z$, and $\mu(\sigma, z) = \mu(\alpha, z)$. We have
\[
	\text{det}^k(\mu(\sigma, z)^{-1}) = j_{\sigma}^k(z)^{-1} := 
		\begin{cases}
			|\mu(\sigma, z)|^{-k} &\text{if $k\in\mathbb{Z}$, $\sigma\in G_{\mathbb{A}}$} \\
			h_{\sigma}(z)^{-1}|\mu(\sigma, z)|^{-[k]} &\text{if $k\not\in\mathbb{Z}$, $\sigma\in\mathfrak{M}$}.
		\end{cases}
\]
where $h_{\sigma}(z)$ is the factor of automorphy defined in [\cite{Sh00}, p.130]. If $k\in\mathbb{Z}$ then make the natural assumption that the projection map $\pr$ and any of its associated lifts
\begin{align*}
	r:G_{\mathbb{A}}\to M_{\mathbb{A}} \\
	r_P:P_{\mathbb{A}}\to M_{\mathbb{A}},
\end{align*}
are all the identity.

The level of the studied forms will be congruence subgroups of the following form
\begin{align*}
	\Gamma &= \Gamma[\mathfrak{b}^{-1}, \mathfrak{bc}] := G\cap D[\mathfrak{b}^{-1}, \mathfrak{bc}] \\
	D[\mathfrak{b}^{-1}, \mathfrak{bc}]:&=Sp_n(\mathbb{R})\prod_p D_p[\mathfrak{b}^{-1}, \mathfrak{bc}] \\
	D_p[\mathfrak{b}^{-1}, \mathfrak{bc}]:&=\{x\in Sp_n(\mathbb{Q}_p)\mid a_x, d_x\in M_n(\mathbb{Z}_p), b_x\in M_n(\mathfrak{b}^{-1}_p), c_x\in M_n((\mathfrak{bc})_p)\}
\end{align*}
for a fractional ideal $\mathfrak{b}$ and an integral ideal $\mathfrak{c}$ under the additional assumption that $2\mid \mathfrak{b}^{-1}$ and $2\mid \mathfrak{bc}$ if $k\not\in\mathbb{Z}$. 

Let $\psi$ be a Hecke character such that $\psi(x_{\infty})^n = \sgn(x_{\infty})^{n[k]}$ and $\psi_p(a) = 1$ for any $a\in\mathbb{Z}_p^{\times}$ with $a\in 1+\mathfrak{c}_p\mathbb{Z}_p$. Then we define the following complex vector spaces
\begin{align*}
	\mathcal{M}_{\rho_k}(\Gamma, \psi) :&= \{f\in \mathcal{M}_{\rho_k}\mid f|_{\rho_k}\gamma = \psi_{\mathfrak{c}}(|a_{\gamma}|)f\ \text{for every}\ \gamma\in \Gamma\} \\
	\mathcal{S}_{\rho_k}(\Gamma, \psi) :&= \mathcal{M}_{\rho_k}(\Gamma, \psi)\cap S_{\rho_k}
\end{align*}
where $\psi_{\mathfrak{c}}(x) = \prod_{p\mid\mathfrak{c}}\psi_p(x_p)$.

Suppose, for some Hecke character $\psi$, that $f, g:\mathbb{H}_n\to V$ satisfy $f|_{\rho_k}\gamma = \psi_{\f{c}}(|a_{\gamma}|)f$ and $g|_{\rho_k}\gamma = \psi_{\f{c}}(|a_{\gamma}|)g$ for all $\gamma\in\Gamma$, and endow $V$ with a Hermitian inner product $\prec\cdot, \cdot\succ$ with respect to which $\rho$ satisfies 
\[
	\prec \rho(M)\cdot, \cdot\succ = \prec \cdot, \rho(\transpose{\bar{M}})\cdot\succ
\]
for any $M\in GL_n(\mathbb{C})$. Then the Petersson inner product of $f$ and $g$ is given by 
\[
	\langle f, g\rangle := \text{Vol}(\Gamma\back\mathbb{H}_n)^{-1}\int_{\Gamma\back\mathbb{H}_n}\prec\rho_k(\sqrt{y})f(z), \rho_k(\sqrt{y})g\succ d^{\times} z
\]
whenever this integral is convergent, where 
\[
	\text{Vol}(\Gamma\back\mathbb{H}_n) := \int_{\Gamma\back\mathbb{H}_n} dz, \hspace{20pt} dz := \bigwedge_{p\leq q}(dx_{pq}\wedge dy_{pq}), \hspace{20pt} d^{\times}z := |y|^{-n-1}dz
\]
for $z = (x_{pq}+iy_{pq})_{p, q = 1}^n.$

With $(\f{b}, \f{c})$ and $\Gamma = \Gamma[\f{b}^{-1}, \f{bc}]$ as above we take $f\in\mathcal{M}_{\rho_k}(\Gamma, \psi)$. Then its adelization is a map $f_{\mathbb{A}}:\pr^{-1}(G_{\mathbb{A}})\to V$ defined by
\[
	f_{\mathbb{A}}(\alpha w) := \psi_{\f{c}}(|d_w|)(f|_{\rho_k}w)(i)
\]
for $\alpha\in G$ and $w\in \pr^{-1}(D[\f{b}^{-1}, \f{bc}])$. Note that if $x\in G_{\mathbb{A}}$ and $\alpha, w$ as above then
\begin{align*}
	f_{\mathbb{A}}(\alpha xw) = \psi_{\f{c}}(|d_w|)\rho_k(\mu(w, z)^{-1})f_{\mathbb{A}}(x).
\end{align*}
Let $t\in \pr^{-1}(G_{\mathbf{h}})$, $\Gamma^t := G\cap tD[\f{b}^{-1}, \f{bc}]t^{-1}$, and set
\[
	\mathcal{M}_{\rho_k}(\Gamma^t, \psi) := \{f\in\mathcal{M}_{\rho_k}\mid f|_{\rho_k}\gamma = \psi_{\f{c}}(|a_{t^{-1}\gamma t}|)f\ \text{for every}\ \gamma\in\Gamma^t\}.
\]

\begin{prop} \label{bijection} For any such $t$ as above and $y = \alpha w\in G\pr^{-1}(D[\f{b}^{-1}, \f{bc}])$ we have $f_{\mathbb{A}}(ty) = (f_t|_{\rho_k}y)(i)$ for some $f_t\in\mathcal{M}_{\rho_k}(\Gamma^t, \psi)$.
\end{prop}

\begin{proof} Heuristically, the function $f_t$ is the translation of $f$ to some cusp. With $t$ fixed as above let $ty = \alpha_tw_t$ where $\alpha_t\in G$ and $w_t\in \pr^{-1}(D[\f{b}^{-1}, \f{bc}])$, then by definition 
\[
	f_{\mathbb{A}}(ty) = \psi_{\f{c}}(|d_{w_t}|)(f|_{\rho_k}w_t)(i)
\]
and so define $f_t$ by letting $y\in \pr^{-1}(G_{\mathbb{A}})$ vary in
\[
	f_t(y\cdot i) = \psi_{\f{c}}(|d_{w_t}|)\rho_k(\mu(y, i))(f|_{\rho_k}w_t)(i).
\]
\end{proof}

The following is an extension of Proposition 20.2 in \cite{Sh00} from the scalar weight to the vector valued case.  

\begin{thm} \label{adelicexp} If $f\in\mathcal{M}_{\rho_k}(\Gamma, \psi)$ then there exists $c_f(\tau, q)\in V$ for $\tau\in S_{+}$ and $q\in GL_n(\mathbb{Q})_{\mathbb{A}}$ such that
\[
	f_{\mathbb{A}}\left(r_P\begin{pmatrix} q & s\transpose{q}^{-1} \\ 0 & \transpose{q}^{-1}\end{pmatrix}\right) = \rho_{[k]}(\transpose{q}_{\infty})\big|\det(q)_{\infty}^{k-[k]}\big|\sum_{\tau\in S_{+}} c_f(\tau, q)e_{\infty}(\tr(i\transpose{q}\tau q))e_{\mathbb{A}}(\tr(\tau s))
\]
for any $s\in S_{\mathbb{A}}$. The coefficients satisfy the following properties:
\begin{enumerate}
	\item $c_f(\tau, q)\neq 0$ only if $e_{\mathbf{h}}(\tr(\transpose{q}\tau qs)) = 1$ for any $s\in S_{\mathbf{h}}(\f{b}^{-1})$;
	\item $c_f(\tau, q) = c_f(\tau, q_{\mathbf{h}})$;
	\item $c_f(\transpose{b}\tau b, q) = \rho_{[k]}(\transpose{b})|\det(b)^{k-[k]}|c_f(\tau, bq)$ for any $b\in GL_n(\mathbb{Q})$;
	\item $\psi_{\mathbf{h}}(|e|)c_f(\tau, qe) = c_f(\tau, q)$ for any $e\in \prod_pGL_n(\mathbb{Z}_p)$.
\end{enumerate}
If $t\in \pr^{-1}(G_{\mathbf{h}})$, $r\in GL_n(\mathbb{Q})_{\mathbf{h}}$, and $\beta\in G\cap \diag[r, \tilde{r}]D[\f{b}^{-1}, \f{bc}]t^{-1}$, then we have
\[
	\rho_k(\mu(\beta, \beta^{-1}z))f_t(\beta^{-1}z) = \psi_{\f{c}}(|a_{\beta t}^{-1}r|)\sum_{\tau\in S_{+}}c_f(\tau, r)e_{\infty}(\tr(\tau z)).
\]
\end{thm}

\begin{proof} Let $x = \begin{pmatrix} q & s\transpose{q}^{-1} \\ 0 & \transpose{q}^{-1}\end{pmatrix}$ as in the theorem, and put $t = x_{\mathbf{h}}$. The functions $f_t~\in~\mathcal{M}_{\rho_k}(\Gamma^t, \psi)$ given in Proposition \ref{bijection} have Fourier expansions
\[
	f_t(z) = \sum_{\tau\in S_{+}} c_f'(\tau)e_{\infty}(\tr(\tau z))
\]
where the coefficients $c_f'(\tau) = c_f'(\tau, q, s)$ depend on $q$ and $s$. We have $x_{\infty}i = \transpose{q}_{\infty}q_{\infty}i+s_{\infty}$ and by Proposition \ref{bijection} that
\begin{align*}	
	f_{\mathbb{A}}(x) = (f_t|_{\rho_k}x)(i) = \rho_{[k]}(\transpose{q}_{\infty})\big|\det(q)_{\infty}^{k-[k]}\big|\sum_{\tau\in S_{+}}c_f'(\tau, q, s)e_{\infty}(\tr(i\transpose{q}\tau q))e_{\infty}(\tr(\tau s)).
\end{align*}
Subsequently defining $c(\tau, q, s) := e_{\mathbf{h}}(-\tr(\tau s))c_f'(\tau, q, s)$ gives us
\[
	f_{\mathbb{A}}(x) = \rho_{[k]}(\transpose{q}_{\infty})\big|\det(q)_{\infty}^{k-[k]}\big|\sum_{\tau\in S_{+}}c(\tau, q, s) e_{\infty}(\tr(i\transpose{q}\tau q))e_{\mathbb{A}}(\tr(\tau s)).
\]
Since $f_{\mathbb{A}}(\alpha xw) = f_{\mathbb{A}}(x)$ for any 
\[
	\alpha = \begin{pmatrix} 1 &\star \\ 0 & 1\end{pmatrix}\in G\ \text{and}\ w =  \begin{pmatrix} 1 &\star \\ 0 & 1\end{pmatrix}\in G_{\mathbf{h}}\cap \pr^{-1}(D[\f{b}^{-1}, \f{bc}])\]
we get independence of the $c_f(\tau, q, s)$ on $s$ as seen in [\cite{Sh00}, p. 168]. This yields our Fourier expansion, and with this the proof of the properties for the coefficients follows through exactly as it does in [\cite{Sh00}, p.168].
\end{proof}

\section{Theta Series}

In this section we obtain some vector valued theta series, which we will employ later to obtain an integral expression of the standard $L$ function attached to an eigenform. These theta series will be Siegel modular forms in $M_{\rho}(\Gamma, \psi)$ for some particular congruence subgroup $\Gamma$ and a nebentype $\psi$, which will become explicit later. However we will not be able to obtain theta series for any weight $\rho$ but rather for some specific ones. The existence of such theta series is closely related to the existence of some pluriharmonic polynomials. We first summarize some results from \cite{KV}. Actually here we restrict to the case that will be of interest to us and do not work in full generality as in their work.

\subsection{Pluriharmonic Polynomials}

Following \cite{KV}, as well as \cite[section 6 and its appendix]{F} we denote by $\mathbb{C}[M_n]$ the ring of complex polynomials on the $n \times n$ matrices. A polynomial $p \in \mathbb{C}[M_n]$ is called pluriharmonic if 
\[
(\Delta_{i,j} p)(x) = 0, \,\,\,1 \leq i \leq j \leq n,
\]
where $\Delta_{i,j} := \sum_{k=1}^n \frac{\partial^2}{\partial x_{ik} \partial_{jk}}$. We denote the space of pluriharmonic polynomials by $\mathfrak{H}$. The group $O_n(\mathbb{C}) \times GL_n(\mathbb{C})$ acts on $\mathbb{C}[M_n]$ by $(g,h) : p(x) \mapsto p(g^{-1}x h)$ and this action preserves the subspace $\mathfrak{H}$. We now consider an irreducible representation $(\lambda,V_{\lambda})$ of $O_n(\mathbb{C})$, and denote by $\mathfrak{H}(\lambda)$ the space of all $V_{\lambda}$ valued pluriharmonic polynomials $p(x)$ such that $p(g x ) = \lambda(g)^{-1}p(x)$. Here pluriharmonicity is understood component-wise.

Following Kashiwara and Vergne we write $\Sigma$ for the set of irreducible $\lambda \in O_n(\mathbb{C})^{\land}$ such that  $\mathfrak{H}(\lambda) \neq 0$, and denote by $\tau(\lambda)$ the representation of $GL_n(\mathbb{C})$ obtained by $\tau(\lambda)(\transpose{g})(p) = p(xg)$, for $p \in \mathfrak{H}(\lambda)$. Then it is shown in \cite{KV} that $\tau(\lambda)$ is an irreducible representation of $GL_n(\mathbb{C})$. Kashiwara and Vergne have determined the representations $\lambda \in \Sigma$ and for each such $\lambda$ they have also described the representation $\tau(\lambda)$. In order to give their results we have to distinguish between the case of $n$ being even or odd.

\subsubsection{The case $n = 2 l + 1$}  Following the notation of \cite{KV} we parametrize the irreducible representations of $O_n(\mathbb{C})$ by the $l + 1$ tuple $(m_1,m_2,\ldots,m_l; \epsilon)$ where $m_j \in \mathbb{Z}$ and $m_1 \geq m_2 \geq \ldots \geq m_l \geq 0$ and $\epsilon = \pm 1$. Here we are using the fact that $O_n(\mathbb{C})~\cong~SO_n(\mathbb{C}) \times \mathbb{Z}/ 2 \mathbb{Z}$, and that $(m_1,\ldots,m_l)$ is the highest weight of a representation of $SO_n(\mathbb{C})$.  Similarly we parametrize the finite dimensional irreducible representations of $GL_n(\mathbb{C})$ by the highest weights with respect to upper triangular matrices, that is by $n$-tuples $(m_1, m_2, \ldots, m_n)$ with $m_1 \geq m_2 \geq \ldots \geq m_n$ with $m_j \in \mathbb{Z}$. We summarize the results of Kashiwara and Vergne in the following theorem.

\begin{thm}[Kashiwara and Vergne] Assume $n = 2l +1$. Every $\lambda \in O_n(\mathbb{C})^{\land}$ belongs to $\Sigma$. Let $\lambda = (m_1,\ldots,m_l; \epsilon)$. If $\epsilon = (-1)^{\sum_{j}m_j}$ then 
\[
\tau(\lambda) = (m_1,\ldots,m_l,0,0,\ldots,0)
\]
If $\epsilon = (-1)^{1+\sum_{j}m_j}$ and we write $\lambda = (m_1,\ldots,m_r,0,\ldots,0; \epsilon)$ with $m_r \neq 0$, then
\[
\tau(\lambda) = (m_1,\ldots,m_{r-1}, m_r, 1,\ldots,1,0, 0, \ldots, 0),
\]
where the sequence ends with $r$-many zeros. If $\lambda = (0,0,\ldots,0;-1)$ then it is understood that $\tau(\lambda) = (1,1,\ldots,1)$.
\end{thm}

\subsubsection{The case $n=2l$} In order to describe the results of Kashiwara and Vergne in this case, we need to introduce some more notation. For an $l$-tuple $(m_1,m_2,\ldots,m_l)$ of decreasing non-negative integers we denote by $\lambda = (m_1,m_2,\ldots,m_l)_+$ the irreducible representation of $O_n(\mathbb{C})$ generated by $\Delta_1(x)^{m_1-m_2} \ldots \Delta_{l-1}^{m_{l-1}-m_l}(x) \Delta_l^{m_l}(x)$ under left translations, where for the definition of $\Delta_j$ we refer the reader to \cite{KV}. Moreover, for an integer $1 \leq j \leq  l$, such that $m_j \neq 0$ and $m_{j+1} = 0$, we denote by $\lambda = (m_1,\ldots,m_l)_-$ the irreducible representation generated by $\Delta_1(x)^{m_1-m_2} \ldots \Delta_{j-1}(x)^{m_{j-1}-m_j} \Delta_j(x)^{m_j-1} \tilde{\Delta}_j(x,y)$ under left translation, where again we refer to \cite{KV} for the definition of $\tilde{\Delta}_j(x,y)$. We only note here that $\tilde{\Delta}_l(x,y) = \Delta_l(x)$ and hence in particular if $m_l \neq 0$ we have that $(m_1,\ldots,m_l)_+ = (m_1,\ldots,m_l)_-$.  We can now state

\begin{thm}[Kashiwara and Vergne] Assume $n = 2l$. Then $\Sigma = \Sigma_+ \cup \Sigma_{-}$ where 
\[
\Sigma_\pm :=   \left\{ (m_1,\ldots,m_l)_{\pm}\right\}
\]
Let $\lambda = (m_1,\ldots,m_l)_{+} \in \Sigma_+$, then 
\[
\tau(\lambda) = (m_1,\ldots,m_l,0,0,\ldots,0)
\]
If $\lambda = (m_1,\ldots,m_r,0,\ldots,0)_{-} \in \Sigma_{-}$ with $m_r \neq 0$, then
\[
\tau(\lambda) = (m_1,\ldots,m_r, 1, 1,\ldots,1,0, 0, \ldots, 0),
\]
where the sequence ends with $r$-many zeros. 
\end{thm}

\subsection{Theta series with pluriharmonic coefficients} Let now $\rho \in \tau(\Sigma)$, that is $\rho = \tau(\lambda)$ for some $\lambda \in \Sigma$. We will construct a theta series with weight $\rho_{n/2} := \rho \otimes \det^{n/2}$ where of course in the case where $n$ is odd, the notation $\det^{n/2}$ is understood as a half-integral weight.

We start by recalling a Jacobi theta series defined by Shimura in \cite{Sh00}. Let $\tau \in M_n(\mathbb{Q})$ be symmetric and positive definite. We also write $W := M_{n}(\mathbb{Q})$ and consider an $\lambda \in \mathcal{S}(W_{\mathbf{h}})$, the Schwarz space of the finite adeles of $W$. Following Shimura \cite[Appendix A.3]{Sh00} we define, for $z\in\mathbb{H}_n$ and $u\in M_n(\mathbb{C})$, the series
\begin{align*}
	g(u,z,\lambda) :&= \sum_{\xi \in W} \lambda(\xi_{\mathbf{h}}) \Phi(\xi;u,z), \\
	\Phi(\xi; u,z) :&= e(\transpose{\tilde{u}} (1_n \otimes 4i y)^{-1} \tilde{u} + 2^{-1} \tr(z \transpose{\xi}\tau\xi) + \tr(u \sqrt{\tau}\xi)) \\
	\tilde{u} :&= \transpose{(u_{11},\ldots,u_{n1},\ldots,u_{1n},\ldots, u_{nn})}.
\end{align*}

The properties of this theta series are stated in \cite[Theorem A3.3]{Sh00}. In particular it is shown there that 
\begin{align}
J(\alpha,z)^{-1} g(\alpha(u,z);  \action{\lambda}) = g(u,z;\lambda),\,\,\,\,\,\alpha \in Sp_n(\mathbb{Q}) \cap \f{M}_n
\end{align}
in which
\begin{align*}
J(\alpha,z) :&= \begin{cases}   j(\alpha,z)^{n/2} &\text{if $n$ is even}, \\
                                                         h(\alpha,z)^n &\text{if $n$ is odd}, \end{cases} \\                                         
\alpha(u,z) :&= (\transpose{\mu(\alpha,z)}^{-1} u, \alpha(z)),
\end{align*}
and where we recall $\mu(\alpha,z) := c_{\alpha}z + d_{\alpha} \in GL_n(\mathbb{C})$. 
Here $\f{M}_n = Sp_n(\mathbb{A})$ if $n$ is even, and is equal to a certain subgroup of the adelic metaplectic group if $n$ is odd (see line (A2.17) of \cite{Sh00}), and in \cite{Sh00} an action of this group is defined on the space $\mathcal{S}(W_{\mathbf{h}})$, which is denoted by $\lambda \mapsto \action{\lambda}$.

Consider a representation $(\rho,V)$ of $GL_n(\mathbb{C})$ with $\rho \in \tau(\Sigma)$, and let $m~:=~\dim_{\mathbb{C}}(V)$. In particular there exists a $V$-valued pluriharmonic polynomial $P(x) = (P_1(x),\ldots,P_m(x))$ with $x \in M_n(\mathbb{C})$ such that $P(x \transpose{g}) = \rho(g) P(x)$ for $g \in GL_n(\mathbb{C})$. Moreover we can select $P(x)$ (see \cite[remark (6.5)]{KV}) to be 
a highest weight vector with respect to the representation $\rho$.

We now define the following $V$-valued theta series,

\[
\theta(z,\lambda; P) := \sum_{\xi \in W} \lambda(\xi_{\mathbf{h}}) P(\sqrt{\tau}\xi) e(2^{-1}\tr(\transpose{\xi} \tau \xi z)).
\]

The following theorem generalizes the one in \cite{Sh00} from the scalar weight situation to the vector valued one. We also refer the reader to \cite{F} for vector valued theta series.

\begin{thm} For any $\alpha \in Sp_n(\mathbb{Q}) \cap \f{M}_n$ we have,
\[
\theta( \alpha \, z, \action{\lambda};P) = J_{\rho}(\alpha,z) \theta(z,\lambda;P).
\]
where
\[
J_{\rho}(\alpha,z) := \begin{cases}   j(\alpha,z)^{n/2} \rho(\mu(\alpha,z)) &\text{if $n$ is even}, \\
                                                         h(\alpha,z)^n\rho(\mu(\alpha,z)) &\text{if $n$ is odd}. \end{cases}
                                                                      \]
\end{thm}

\begin{proof} We consider the differential operator $P(\partial) = (P_1(\partial),\ldots,P_m(\partial))$ on the space $M_n(\mathbb{C})$, where we have set $x_{ij} = \frac{\partial}{\partial u_{ij}}$. Here for a function $f(u)$ on $M_n(\mathbb{C})$ we understand that $P(\partial) f := (P_1(\partial)f, \ldots, P_m(\partial)f)$ is a $V$-valued function on $M_n(\mathbb{C})$.  We now observe that
\[
2\pi i\theta(  z, \lambda;P)  = (P(\partial) g(u,z;\lambda))|_{u = 0}
\]
Indeed this follows from the fact that $P(\partial) e^{ 2\pi i tr(ua)} = 2\pi i P(a) e^{\pi i tr(ua)}$ for any matrix $a \in M_{n}(\mathbb{C})$ and that
\begin{align*}
	P(\partial)e(\transpose{\tilde{u}} (1_n \otimes 4i y)^{-1} \tilde{u} + &2^{-1}\tr(z\transpose{\xi}\tau\xi) + \tr(u \sqrt{\tau}\xi))|_{u=0} \\
	&=P(\partial)(e( 2^{-1} \tr(z \transpose{\xi}\tau\xi) + \tr(u \sqrt{\tau}\xi)  )|_{u=0}
\end{align*}
by \cite[Lemma A3.6]{Sh00}. We now apply to $(3)$ above the operator $P(\partial)(\cdot)|_{u=0}$ to both sides and, observing that $P(\sqrt{\tau}\xi \transpose{\mu(\alpha,z)}^{-1}) = \rho(\mu(\alpha,z))^{-1}P(\sqrt{\tau}\xi)$, we conclude the proof.
\end{proof}

The function $\theta(z,\lambda, P)$ enjoys the same properties with respect to level as the function $\theta(z,\lambda)$ defined in \cite{Sh00}, since they are both obtained by the Jacobi theta series by applying differential opeartors. Actually it is exactly the same function if we take $\rho = \det^{\mu}$ where $\mu = \{0,1\}$, since $\det$ is in $\tau(\Sigma)$. In particular Propositions A. 3.17 and A. 3.19 in \cite{Sh00} hold also for the theta function defined here when one replaces $\det(\xi)^\mu$ there with $P(\sqrt{\tau}\xi)$. Indeed the level of the theta series depends, thanks to the theorem above, only on the choice of the Schwartz function $\lambda$. We now describe a particular choice of $\lambda$ and give the congruence subgroup of the corresponding theta series. 

We start with a Hecke character $\chi$ of conductor $\mathfrak{f}$.  For a fixed $Q\in GL_n(\mathbb{Q}_{\mathbf{h}})$, we define the theta seres $\theta_{\rho,\chi}(z):= \theta(  z, \lambda;P)$ where the Schwartz function $\lambda$ is given at each place by $\lambda_p(x) = \chi_p^{-1}(|Q|)\lambda_p'(Q^{-1}x)$,
\[
	\lambda_p'(x) := \begin{cases} 1 & \text{if $x\in M_n(\mathbb{Z}_p)$ and $p\nmid\mathfrak{f}$} \\ \chi_p^{-1}(|x|) &\text{if $x\in GL_n(\mathbb{Z}_p)$ and $p\mid\mathfrak{f}$} \\ 0 &\text{otherwise}, \end{cases}
\]
and, overall, by $\lambda(x) := \prod_p\lambda_p(x_p)\in\mathcal{S}(M_n(\mathbb{Q}_{\mathbf{h}}))$. Then as in the scalar weight case (see Proposition A3.19 in \cite{Sh00}) we have that $\theta_{\rho,\chi}(z) \in \mathcal{M}_{\rho_{n/2}}(\Gamma,\chi\epsilon_{\tau})$ where $\epsilon_{\tau}$ is the quadratic character, of conductor $\f{h}$, corresponding to the extension $F(i^{[n/4]}\sqrt{|2\tau|})/F$; $\Gamma = G \cap D[\mathfrak{b}^{-1}, \mathfrak{bc}]$ for a fractional ideal $\f{b}$ and integral ideal $\f{c}$ given by
\[
(\mathfrak{b},\mathfrak{c}) = \begin{cases} (2^{-1} \mathfrak{r} ,\mathfrak{h} \cap \mathfrak{f} \cap \mathfrak{r}^{-1} \mathfrak{f}^2 \mathfrak{t}) &\text{if $n \in 2 \mathbb{Z}$} \\ 
                                                                      (2^{-1} \mathfrak{a}^{-1} ,\mathfrak{h} \cap \mathfrak{f} \cap 4 \mathfrak{a} \cap \mathfrak{a} \mathfrak{f}^2 \mathfrak{t}) &\text{otherwise},   \end{cases}
\]
in which the ideals $\mathfrak{r}$, $\mathfrak{t}$, and $\f{a}$ are defined, for all $g \in Q \mathbb{Z}^n$ and for all $h \in \transpose{Q}^{-1} \mathbb{Z}^n$, by 
\begin{align*}
	\transpose{g} 2 \tau g &\in \mathfrak{r}, \\
	 \transpose{h}(2 \tau)^{-1} h &\in 4 \mathfrak{t}^{-1}, \\
	\mathfrak{a} :&= \mathfrak{r}^{-1} \cap \mathbb{Z}.
\end{align*}

\subsection{Cuspidal theta series.} Our next aim in this section is to obtain a result towards the cuspidality of  this theta series, which will be useful later in establishing non-vanishing results for the $L$-function of a cusp form.  We first note that Theorem A3.3 (5) and (6) of \cite{Sh00} tell us that if $\sigma\in\f{M}_n$ is such that $\pr(\sigma) = \begin{pmatrix} 1 & b_{\sigma} \\ 0 & 1\end{pmatrix}$, $x\in M_n(\mathbb{Q})$, and $\tau\in S_+$ is fixed, then
\begin{align}	
	(^{\sigma}\lambda)(x) &= e_{\mathbf{h}}\left(\tr(\transpose{x}\tau x \transpose{b}_{\sigma})\right)\lambda(x) \\
	(^{\eta}\lambda)(x) &= \int_{M_n(\mathbb{Q}_{\mathbf{h}})} \lambda(y) e_{\mathbf{h}}(-\tr(\transpose{x}2\tau y))d_{\mathbf{h}}y,
\end{align}
where recall $\eta := \begin{pmatrix} 0 & -1_n \\ 1_n & 0\end{pmatrix}$ and $d_{\mathbf{h}}y = \prod_p d_py$ is the Haar measure such that $\int_{M_n(\mathbb{Z}_p)}d_py = |\det(2\tau)|_p^{\frac{n}{2}}$. 

If $\chi$ is a character modulo $\f{f} = F\mathbb{Z}$, $X\in M_n(\mathbb{Z})$, and $R\in S$ is a symmetric matrix, then define the generalised quadratic Gauss sum by
\[
	G_n'(\chi, X, R, F) := \sum_{T\in M_n(\mathbb{Z}/F\mathbb{Z})}\chi(|T|)e^{2\pi i\frac{\tr(\transpose{X}T-\tau[Q]TR\transpose{T})}{F}},
\]
where $\tau[Q] := \transpose{Q}\tau Q$. The integral $(^{\sigma\eta}\lambda)(x)$ is calculated as follows.

\begin{lem}\label{integralgauss} Let $\chi$ be a character modulo $\f{f} = F\mathbb{Z}$ and put $F_p :=\ord_p(\f{f})$. Assume that $\sigma$ and $\tau$ are as above, that $b = b_{\sigma}\in S(\mathbb{Z}_p)$, and let $Q\in GL_n(\mathbb{Q}_{\mathbf{h}})$. Then the value of $(^{\eta\sigma}\lambda)(x)$ is non-zero if and only if 
\[
	p^{F_p}\tau[Q] - 2\transpose{x}\tau Q\in \begin{cases}p^{-F_p}M_n(\mathbb{Z}_p/p^{F_p}\mathbb{Z}_p) &\text{if $p\mid\f{f}$} \\ M_n(\mathbb{Z}_p) &\text{if $p\nmid\f{f}$} \end{cases}
\]
at which it is given by
\[
	(^{\eta\sigma}\lambda)(x) = \chi(|Q|)|2QF\tau|^{-\frac{n}{2}}G_n'(\chi, 2F \transpose{Q}\tau x, Fb, F).
\]
\end{lem}

\begin{proof} First consider the local integrals for $p\mid\f{f}$, at which we have that $(^{\eta\sigma}\lambda)_p(x)$ is equal to
\begin{align*}
	\chi_p^{-1}(|Q|)&\int_{QGL_n(\mathbb{Z}_p)} \chi_p^{-1}(|Q^{-1}y|)e_p\left(\tr(\tau yb\transpose{y} - \transpose{x}2\tau y)\right)d_py \\
	&= \chi_p^{-1}(|Q|)|\det(Q)|_p^{\frac{n}{2}}\int_{GL_n(\mathbb{Z}_p)}\chi_p^{-1}(|y|)e_p\left(\tr(\tau[Q]yb\transpose{y}-\transpose{x}2\tau Qy)\right)d_py.
\end{align*}

Since the local conductor of $\chi_p$ is $p^{F_p}$ this becomes
\begin{align*}
	\chi_p^{-1}(|Q|)|\det(Q)|_p^{\frac{n}{2}}\sum_{a\in M_n(\mathbb{Z}/p^{F_p}\mathbb{Z})}&\chi_p^{-1}(|a|)e\left(\tr(\transpose{x}2\tau Qa - \tau[Q]ab\transpose{a})\right) \\
	&\times\int_{M_n(\mathbb{Z}_p)}e_p\left(\tr(p^{2F_p}\tau[Q]yb\transpose{y} - p^{F_p}\transpose{x}2\tau Qy)\right)d_py.
\end{align*}
The integral on the second line is non-zero if and only if the integrand is a constant function in $y$ -- i.e. if and only if $p^{F_p}\tau[Q]- 2\transpose{x}\tau Q\in p^{-F_p}M_n(\mathbb{Z}_p)$ -- at which point it is equal to $p^{-F_p(n^2/2)}$. Multiplying all such local sums together for $p\mid\f{f}$ gives the form in the lemma. Note that if $p^{F_p}\tau[Q] - 2\transpose{x}\tau Q\in M_n(\mathbb{Z}_p)$ then the above expression becomes a sum of a character over all its values, which is zero. 

When $p\nmid\f{f}$ then the local integral $(^{\eta\sigma}\lambda)_p(x)$ is equal to
\[
	\int_{QM_n(\mathbb{Z}_p)} e_p\left(\tr(\tau yb\transpose{y} - \transpose{x}2\tau y)\right)d_py
\]
and this is non-zero if and only if we have the condition given in the lemma at which point, by the definition of the Haar measure, it is $|\det(2\tau Q)|_p^{\frac{n}{2}}$.
\end{proof}

\begin{prop}\label{nquadgauss} If $\det(X)=0$; $p$ is an odd prime; $\tau = \diag[\tau_1, \dots, \tau_n]$ is diagonal, $Q\in GL_n(\mathbb{Q}_{\mathbf{f}})$ is upper triangular with coefficients $q_{n1}, \dots, q_{n, n-1} = 0$; and $\chi$ is odd of conductor $p$, then
\[
	G_n'(\chi, X, R, p) = 0.
\]
\end{prop}

\begin{proof} In the base $n=1$ case, $0 = X\in\mathbb{Z}$ and we can write
\[
	G_1'(\chi, X, R, p) = \sum_{n\in\mathbb{F}_p^{\times}} \chi(n)e^{-2\pi i\tau Q^2\frac{Sn^2}{p}} = \sum_{n=1}^{\frac{p-1}{2}}[\chi(n)+\chi(-n)]e^{-2\pi i\tau Q^2\frac{Rn^2}{p}}
\]
for $Q\in\mathbb{Q}$; $\tau, R\in\mathbb{Z}$; and this is zero if $\chi$ is odd.

For the general $n$ case, first let $M_{jn}$ be the $(n-1)\times (n-1)$ matrix obtained from any $n\times n$ matrix $M$ by removing the $j$th row and the $n$th column. By a change of basis followed by a change of variables in $T$ we can assume that
\[
	X = \begin{pmatrix} X_{nn} & \mathbf{0} \\ \mathbf{x} & 0\end{pmatrix},
\]
where $\begin{pmatrix}\mathbf{x} & 0 \end{pmatrix}\in\mathbb{Z}^n$ is the $n$th row of $X$. Let $\mathbf{t}_i$ be the $i$th column of $T$, and let $\mathbf{t}_i^Q$ be the $i$th column of $T\transpose{Q}$. Then
\begin{align*}
	\tr(\transpose{X}T) &= \tr(\transpose{X}_{nn}T_{nn}) + \sum_{i=1}^{n-1}x_{ni}t_{ni} \\
 	\tr(\tau[Q]\transpose{T}RT) &= \tr (\tau \transpose{(T\transpose{Q})}R T\transpose{Q}) = \sum_{i=1}^n \tau_i\transpose{\mathbf{t}}_i^QR\mathbf{t}_i^Q
\end{align*}
and so within the sum defining $G_n'(\chi, X, R, p)$ appears the following subsum
\begin{align}
	\sum_{\mathbf{t}_n\in\mathbb{F}_p^n} \chi\left(|T|\right) e\left(-\tfrac{\tau_nq_n^2\transpose{\mathbf{t}}_nR\mathbf{t}_n}{p}\right).
\end{align}
We have been able to separate the variables as such by the specific form of $Q$ and by using that $\mathbf{t}_n^Q = q_n\mathbf{t}_n$ since $Q$ is upper triangular. The proof is completed by showing that the above sum (6) is zero if $\chi$ is odd. By Lemma A1.5 of \cite{Sh00}, there exists $\alpha\in GL_n(\mathbb{Z})$ such that $\transpose{\alpha}^{-1}R\alpha^{-1} = R' := \diag[r_1, \dots, r_n]$ is diagonal. Using the expansion
\[
	|T| = \sum_{j=1}^n(-1)^{n+j}t_{jn}|T_{jn}|
\]
the sum of (6) can be written as
\[
	\chi(|\alpha|^{-1})\sum_{\mathbf{t}_n\in\mathbb{F}_p^n} \chi\left(|\alpha T|\right) e\left(-\tfrac{\tau_nq_n^2\transpose{(\alpha\mathbf{t}_n)}R'(\alpha\mathbf{t}_n)}{p}\right)
\]
and this right-hand sum, after a change of variables, becomes
\begin{align}
	\sum_{(t_{1n}, \dots, t_{nn})\in\mathbb{F}_p^n} \chi\left(\sum_{j=1}^n(-1)^{n+j}t_{jn}|(\alpha T)_{jn}|\right) e\left(-p^{-1}\tau_nq_n^2\sum_{j=1}^n r_jt_{jn}^2\right) .
\end{align}
In the base $n=1$ case (7) becomes $G_1'(\chi, X, R, p)$ which we have shown to be zero at the beginning of this proof. So now assume that the $n-1$ degree sum corresponding to (7) is zero. If one of the $t_{jn} = 0$ in (7), then it becomes the $n-1$ degree sum and is therefore zero. So we can assume by induction that $(t_{1n}, \dots, t_{nn})\in (\mathbb{F}_p^{\times})^n$, which set can be partitioned as
\[
	(\mathbb{F}_p^{\times})^n = \bigsqcup_{i_1, \dots, i_n = 0}^1 (-1)^{i_j}(\mathbb{F}_p^-)^n
\]
for $\mathbb{F}_p^- := \{1, \dots, \frac{p-1}{2}\}$. This can easily be seen by writing any $(a_1, \dots, a_n)\in(\mathbb{F}_p^{\times})^n$ as $((-1)^{i_1}a_1', \dots, (-1)^{i_n}a_n')$, where $a_j' = |a_j''|$ and $a_j''$ is the representative of $a_j$ taken in $\{\pm 1, \dots, \pm\frac{p-1}{2}\}$. The aim is to re-write the sum of (7) over $(\mathbb{F}_p^-)^n$. To this end notice that as $(t_{1n}, \dots, t_{nn})\mapsto ((-1)^{i_1}t_{1n}', \dots, (-1)^{i_n}t_{nn}')$ we have
\begin{align*}
	(7)&\mapsto \sum_{(t_{1n}', \dots, t_{nn}')\in(\mathbb{F}_p^-)^n}\sum_{\mathbf{i}\in\mathbb{F}_2^n}\chi(|T|_{\mathbf{i}})e\left(-p^{-1}\tau_nq_n^2\sum_{j=1}^n r_j(t_{jn}')^2\right) \\
	|T|_{\mathbf{i}} :&= \sum_{j=1}^n(-1)^{n+j+i_j}t_{jn}'|(\alpha T)_{jn}|
\end{align*}
where $\mathbf{i} = (i_1, \dots, i_n)$. The argument of the exponential is unchanged by the transformation $((-1)^{i_1}t_{1n}', \dots, (-1)^{i_n}t_{nn}')\mapsto - ((-1)^{i_1}t_{1n}', \dots, (-1)^{i_n}t_{nn}')$, yet in the coefficients we see $|T|_{\mathbf{i}}\mapsto -|T|_{\mathbf{i}}$. Hence we can pair up the coefficients of the exponential as follows. Let $\sim$ be an equivalence relation on $\mathbb{F}_2^n$ defined by $\mathbf{i}_1\sim \mathbf{i}_2$ if and only if $\mathbf{i}_1 = \mathbf{i}_2+\mathbf{1}$. Then (7) becomes
\[
	\sum_{(t_{n1}', \dots, t_{nn}')\in(\mathbb{F}_p^-)^n}\sum_{\mathbf{i}\in\mathbb{F}_2^n/\sim}\left[\chi(|T|_{\mathbf{i}})+\chi(-|T|_{\mathbf{i}})\right]e\left(-p^{-1}\tau_nq_n^2\sum_{j=1}^nr_j(t_{jn}')^2\right)
\]
which is zero, since $\chi$ is odd.	
\end{proof}

\begin{thm}\label{theta} Let $\chi$ be an odd non-trivial Dirichlet character of square free conductor prime to $2$. Then there are choices of $\tau \in S_+$ and $Q \in GL_n(\mathbb{Q}_{\mathbf{h}})$ such that the corresponding $\theta_{\rho,\chi}(z)$ is a cusp form.
\end{thm}  
\begin{proof} We write $\theta(z,\lambda; P)$ for $\theta_{\rho,\chi}(z)$, where $\lambda$ is the corresponding to $\chi$ Schwartz function and $p$ the pluriharmonic polynomial. We first note that for any $\alpha \in Sp_n(\mathbb{Q})$ we have that $\theta(z,\lambda; P)|_{\rho_{\frac{n}{2}}} \alpha = \theta(z, \actioninv{\lambda}; P)$. If we write $\Gamma$ for the congruence subgroup of this theta series, then in order to prove that it is a cusp form, it is enough to show that (see for example \cite[Lemma 27.3]{Sh00})
\[
\Phi\left( \theta |_{\rho_{\frac{n}{2}}} \alpha \right) = 0
\]
where $\alpha$ runs over a set of representatives of $\Gamma \setminus Sp_n(\mathbb{Z}) / P_{n-1}(\mathbb{Z})$ and $P_{n-1}$ denotes the Klingen parabolic corresponding to boundary components of degree $n-1$. For a definition we refer, for example, to \cite[page 595]{Sh95} where it is denoted as $P^{n,n-1}$. Furthermore $\Phi$ denotes Siegel's $\Phi$-operator, a definition of which can be found in \cite[page 219]{Sh00}.

Our aim is to find, explicitly, some representatives for the above double coset. We do this for a congruence subgroup $\Gamma$ of a particular type, namely $\Gamma[m,m]$ where $m$ is a square free integer, i.e. $m = \prod_{i} p_i$ where $p_j \neq p_k$ for $k\neq j$ and $p_i$ primes. Here we denote,
\[
\Gamma[m,m] = \left\{ \gamma = \left(\begin{matrix}  a & b \\ c & d \end{matrix}\right) \in Sp_n(\mathbb{Z}) \bigg| b \equiv 0 \pmod{m},\,\,\,\,c \equiv 0 \pmod{m} \right\}
\]
Our approach is inspired by a similar one done in \cite[page 76]{BP}, where the case of groups $\Gamma_0(m)$ for square free $m$ was considered.
We first consider the case where $m = p$ for some prime $p$. By taking the projection $Sp_n(\mathbb{Z}) \rightarrow Sp_n(\mathbb{F}_p)$ modulo $p$, which is surjective, and since the kernel belongs to $\Gamma(p,p)$ we see that is enough to find representatives for the set $ C:= Q(\mathbb{F}_p) \setminus Sp_n(\mathbb{F}_p) / P_{n-1}(\mathbb{F}_p)$, where $Q(\mathbb{F}_p)$ is the set of matrices $\diag[a, \transpose{a}^{-1}]$ with $a \in GL_n(\mathbb{F}_p)$.  If we write $P_0(\mathbb{F}_p)$ for the points of the Siegel parabolic over the finite field $\mathbb{F}_p$ then we have the Bruhat decomposition
\[
Sp_n(\mathbb{F}_p) = P_0(\mathbb{F}_p) P_{n-1}(\mathbb{F}_p) \cup P_0(\mathbb{F}_p) \eta P_{n-1}(\mathbb{F}_p).
\]
 But then if we use the fact that $P_0 = Q R$ where $R(\mathbb{F}_p) = \left\{ m(s) := \left( \begin{matrix} 1 & s \\ 0 & 1  \end{matrix}\right) \bigg| \transpose{s} = s \right\}$ we can conclude that a set of representatives for the set $C$ can be given by a subset of the matrices $\{ m(s), m(s) \eta : s \in S(\mathbb{F}_p) \}$. By lifting back this set to $Sp_n(\mathbb{Z})$ we obtain a set of representatives for the $n-1$-degree cusps, in the case of $m$ being a prime. 
 
 For the general case, where $m$ is a product of distinct primes we can use the Chinese reminder theorem to show that $Sp_n(\mathbb{Z}/ m \mathbb{Z}) = Sp_n(\mathbb{F}_{p_1}) \times \ldots \times Sp_n(\mathbb{F}_{p_m})$ to reduce everything to the case of a single prime. 
 
 We now explain how we can construct theta series whose congruence subgroup is of the form $\Gamma[2p,2p]$ for an odd prime $p$. We let $\chi$ be a character of conductor $p$. Following the notation above we choose our $\tau$ and $Q$ such that $(\mathfrak{b}, \mathfrak{c}) = ((2p)^{-1}, 4p^2)$. This can be done for example by selecting $\tau = 2pI_n$, and $Q = (2p)^{-1} I_n$. 
 
 With these choices, we now show that the corresponding theta series is cuspidal. Since for any $\alpha \in Sp_n(\mathbb{Q})$ we have that $\theta(z,\lambda; P)|_{\rho_{\frac{n}{2}}} \alpha = \theta(z, \actioninv{\lambda}; P)$, it is enough to show that the support of the Schwartz function $\actioninv{\lambda}$ is on full rank matrices for all the representatives $\alpha$ of the double coset, which we have listed. This can be achieved by using the reciprocity laws, see (4) and (5) above, of the action of the representatives of the cusps above to the Schwarz function of the theta series. We distinguish the cusps according how are represented locally at places $(2, p)$ as follows  
 \begin{align}
 	(m(s_1),m(s_2)), \\
	(m(s_1), m(s_2)\eta), \\
	(m(s_1)\eta,m(s_2)), \\
	(m(s_1)\eta,m(s_2) \eta).
 \end{align}
As is done in \cite{BP} it's enough to check the Schwarz functions locally. For the first kind (8), we have by Theorem A3.3 (5) of \cite{Sh00} that the support of the Schwarz function $\actioninv{\lambda}$ at the corresponding cusp is at the non-singular matrices. For the cusp of (10), this is also clear since the support of $\chi$ is away from 2, and so $^{\eta m(-s_1)}\lambda_2$ is just Theorem A3.3 (5). For the kinds (9) and (11), the Schwartz function $^{\eta m(-s_2)}\lambda_p$ is zero on singular matrices by the lemma and proposition preceding this theorem.
\end{proof}

\section{Rankin-Selberg Integral Representation}

The main aim of this section is to extend some well-known results of the Rankin-Selberg integral expression from the scalar weight case (as for example in \cite{Sh00}) to the vector valued case. From the rest of the paper we will assume that the representation $\rho$ is in $\tau(\Sigma)$. 

\subsection{The function $D(s,f,g)$}

For $0<n\in \mathbb{Z}$ and $k, \ell\in\frac{1}{2}\mathbb{Z}$ we let $f\in S_{\rho_k}(\Gamma_1, \psi_1)$ and $g\in M_{\rho_{\ell}}(\Gamma_2, \psi_2)$ where $\Gamma_1 = \Gamma[\f{b}_1^{-1}, \f{b}_1\f{c}_1], \Gamma_2 = \Gamma[\f{b}_2^{-1}, \f{b}_2\f{c}_2]$ are two congruence subgroups ($\Gamma_1$ or $\Gamma_2$ contained in $\f{M}$ if $k\notin\mathbb{Z}$ or $\ell\notin\mathbb{Z}$ respectively), and $\psi_1, \psi_2$ are two nebentypes such that
\[
 (\psi_1 \bar{\psi}_2)_{\infty}(-1) = (-1)^{[k] - [\ell]}.
\]

Then the non-holomorphic scalar Eisenstein series of weight $k-\ell$ is defined, for $z\in\mathbb{H}_n$ and $s\in\mathbb{C}$, by
\[
	E_{k-\ell} (s)= E_{k-\ell}^n(z, s; \chi, \Gamma) := \sum_{\gamma\in P\cap\Gamma\setminus\Gamma}\chi(a_{\gamma}) j_{\gamma}^{k-\ell}(z)^{-1}\Im(\gamma\cdot z)^{s-\frac{k-\ell}{2}}
\]
where $\Gamma := \Gamma[\f{z}^{-1}, \f{zy}]$, $\f{z} := \f{b}_1+\f{b}_2$, $\f{y} = \f{z}^{-1}(\f{b}_1\f{c}_1\cap\f{b}_2\f{c}_2)$, and $\chi := \bar{\psi}_1\psi_2$. Then we have that $\Vol(\Gamma\back\mathbb{H}_n)\left\langle f, gE_{k-\ell}(s+\tfrac{n+1}{2}) \right\rangle$ is equal to
\begin{align*}
	\int_{\Gamma\setminus\mathbb{H}_n}\sum_{\gamma\in P\cap\Gamma\setminus\Gamma}\prec \rho_k(y)f(z), g(z)(\bar{\psi}_1\psi_2)(|a_{\gamma}|)j_{\gamma}^{k-\ell}(z)^{-1}\succ\Im(\gamma\cdot z)^{s+ \frac{n+1-k+\ell}{2}}d^{\times}z
\end{align*}
where we used that $\prec\rho_k(\sqrt{y})v_1,v_2\succ\ =\ \prec v_1, \rho_k(\sqrt{y}) v_2\succ$ for $v_1,v_2 \in V$. Note that for any $\gamma = \begin{pmatrix} a & b \\ c & d\end{pmatrix}\in\Gamma$ we have
\begin{align*}
	y &= (\transpose{c}\bar{z}+\transpose{d})\Im(\gamma\cdot z)(cz+d) \\
	f(z) &= \bar{\psi}_1(|a|)\rho_k((cz+d)^{-1})f(\gamma\cdot z) \\
	g(z) &= \bar{\psi}_2(|a|)\rho_{\ell}((cz+d)^{-1})g(\gamma\cdot z) \\
	j_{\gamma}^{k-\ell}(z)^{-1}\rho_k(cz+d)\rho_{\ell}((cz+d)^{-1}) &= \rho(cz+d)\rho((cz+d)^{-1}) = 1
\end{align*}
and making these substitutions now gives that $\Vol(\Gamma\back\mathbb{H}_n)\left\langle f, gE_{k-\ell}(s+\tfrac{n+1}{2}) \right\rangle$ is equal to
\begin{align*}
	&\int_{\Gamma\setminus\mathbb{H}_n}\left[\sum_{\gamma\in P\cap\Gamma\setminus\Gamma}\prec\rho_k(\transpose{c}\bar{z}+\transpose{d})\rho_k(\Im(\gamma\cdot z))f(\gamma\cdot z), \rho_{\ell}((cz+d)^{-1})g(\gamma\cdot z)j_{\gamma}^{k-\ell}(z)^{-1}\succ\right. \\
	&\hspace{70pt}\times\Im(\gamma\cdot z)^{s+\frac{n+1-k+\ell}{2}}\bigg]d^{\times}z \\
	&=\int_{\Gamma\setminus\mathbb{H}_n}\sum_{\gamma\in P\cap\Gamma\setminus\Gamma}\prec\rho_k(\Im(\gamma\cdot z))f(\gamma\cdot z), g(\gamma\cdot z)\succ\Im(\gamma\cdot z)^{s+ \frac{n+1-k+\ell}{2}}d^{\times}z
\end{align*}
using that $\prec\rho_k(\transpose{c}\bar{z}+\transpose{d})v_1, v_2\succ\ =\ \prec v_1, \rho_k(c z + d)v_2\succ$ for any $v_1,v_2 \in V$.

If we put $\varphi(z) := (\rho_k(y)f(z), g(z))|y|^{s+ \frac{n+1-k+\ell}{2}}$ then this is $P\cap\Gamma$-invariant. Indeed, this follows from that fact that for a $\gamma \in P$ we have 
\[
\rho_k (\Im(\gamma z)) = \rho_k(\transpose{d}_{\gamma}^{-1}) \rho_k(y) \rho_k(d_{\gamma}^{-1}),\,\,\,\,|\text{Im}(\gamma z)| = |d_{\gamma}|^{-2} |y| 
\]

We can now apply to it the standard unfolding procedure
\[
	\int_{\Gamma\setminus\mathbb{H}_n}\sum_{\gamma\in P\cap\Gamma\setminus\Gamma}\varphi(\gamma\cdot z)|y|^{-n-1}dxdy = \int_{P\cap\Gamma\setminus\mathbb{H}_n} \varphi(z)|y|^{-n-1}dxdy
\]
to obtain
\[
	\Vol(\Gamma\back\mathbb{H}_n)\langle f, gE_{k-\ell}(s+\tfrac{n+1}{2})\rangle = \int_{P\cap\Gamma\setminus\mathbb{H}_n}\prec\rho_k(y)f(z), g(z)\succ |y|^{s+\frac{n+1-k+\ell}{2}}d^{\times}z.
\]
We can take the domain $P\cap\Gamma[1, 1]\setminus\mathbb{H}_n = X\cup Y$ where
\begin{align*}
	X :&= \{x\in M_n(\mathbb{R})\mid x = \transpose{x}\pmod{1}\} \\
	Y :&= \{y\in M_n(\mathbb{R})\mid \transpose{y}=y>0\},
\end{align*}
and the domain $P\cap\Gamma\setminus\mathbb{H}_n$ is $N(\f{z}^{-1})^{\frac{n(n+1)}{2}}$ copies of these. Defining the differentials
\[
	dx := \bigwedge_{p\leq q} dx_{pq}, \hspace{20pt} dy := \bigwedge_{p\leq q} dy_{pq}, \hspace{20pt} d^{\times}y = |y|^{-\frac{n+1}{2}}dy
\]
and using the Fourier expansions of $f$ and $g$ the integral $\Vol(\Gamma\back\mathbb{H}_n)\langle f, gE_{k-\ell}\rangle$ becomes
\begin{align*}
	N(\f{z})^{-\frac{n(n+1)}{2}}\sum_{R, S\in S_+}&\left[\int_Xe(\tr((R-S)x))dx\right. \\
&\left.\times\int_Y\prec \rho_k(y)c_f(R, 1), c_g(S, 1)\succ|y|^{s+\frac{\ell - k}{2}}e^{-4\pi\tr((R+S)y)}d^{\times}y\right].
\end{align*}
The integral over $X$ is only non-zero for $R = S$, at which point it is equal to 1. With a factor of 2 to account for the action of $-I_n$ we obtain the expression
\begin{align}
	 2N(\f{z})^{-\frac{n(n+1)}{2}}\sum_{R\in S^+}\int_Y\prec \rho(y)c_f(R, 1), c_g(R, 1)\succ |y|^{s+\frac{k+\ell}{2}}e^{-4\pi\tr(Ry)}d^{\times}y
\end{align}
for $\Vol(\Gamma\back\mathbb{H}_n)\langle f, gE_{k-\ell}(s+\tfrac{n+1}{2})\rangle$.
Now we set $h:= \frac{k+\ell}{2}$ and define
\[
	H_{\rho, h,R}^{n}(s) = H_{\rho, R}(s) := \int_Y \rho(y)e^{-4 \pi\tr(Ry)}|y|^{s+h}d^{\times}y
\]
By Theorem 3 in Godement's Expose 6, \cite{G57} we have that this operator is Hermitian, and
\[
H_{\rho, R}(s) = \rho(R^{-1/2}) H_{\rho}(s) \rho(R^{-1/2})\det(R)^{-(s+h)}
\]
where $H_{\rho}(s) = H_{\rho, I_n}(s)$.

Plugging this back into (12) we get
\[
	\Vol(\Gamma\back\mathbb{H}_n)\langle f, gE_{k-\ell}(s+ \tfrac{n+1}{2})\rangle = 2N(\f{z})^{-\frac{n(n+1)}{2}}\sum_{R\in S^+} \prec H_{\rho, R}(s)  c_f(R, 1), c_g(R, 1)\succ
\]
and so define the Rankin product of $f$ and $g$ by
\[
	D(s, f, g) := \sum_{R \in \mathcal{S}}\nu_R \prec H_{\rho, R}(s)  c_g(R, 1), c_g(R, 1)\succ
\]
where $\mathcal{S} := S^+/ GL_n(\mathbb{Z})$ and  $\nu_R^{-1} := \sharp \{ u \in GL_n(\mathbb{Z}) : \transpose{u} R u = R\}$. That this is well-defined is shown in the following calculation. Let $u \in GL_n(\mathbb{Z})$, then by definition 
\begin{align}
	H_{\rho, \transpose{u}Ru}(s) = \rho(u^{-1})H_{\rho, R}(s)\rho(\transpose{u}^{-1}). 
\end{align}
With this and Theorem \ref{adelicexp} (3) -- (4) we have that
\begin{align*}
\prec &H_{\rho,\transpose{u}R u}(s)  c_f(\transpose{u}Ru, 1), c_g(\transpose{u}Ru, 1)\succ \\
&=\prec \rho(u^{-1}) H_{\rho, R}^n(s) \rho(\transpose{u}^{-1} )  \rho_{[k]}(\transpose{u}) \psi_1(|u|)c_f(R, 1), \rho_{[\ell]}(\transpose{u})\psi_2(|u|)c_g(R, 1)\succ \\
&=\ \prec H_{\rho,h,R}^n(s)c_f(R, 1), c_g(R, 1)\succ (\psi_1\bar{\psi}_2)(|u|)|u|^{[\ell]-[k]}|u|^{2[k]} \\ 
&=\ \prec H_{\rho,h,R}^n(s)c_f(R, 1), c_g(R, 1)\succ.
\end{align*}
That is, we can now conclude
\begin{equation}
D(s, f, g) = 2^{-1}N(\f{z})^{\frac{n(n+1)}{2}}\langle f, gE_{k-\ell}(s+ \tfrac{n+1}{2})\rangle \Vol(\Gamma\back\mathbb{H}_n).
\end{equation}

The following result is a generalization of Proposition 22.2 in \cite{Sh00} from the scalar weight case to the vector valued case. 

\begin{prop} \label{dirichletconv} With $k, \ell\in\frac{1}{2}\mathbb{Z}$, $f\in S_{\rho_k}(\Gamma_1, \psi_1)$, and $g\in M_{\rho_{\ell}}(\Gamma_2, \psi_2)$ such that $(\psi_1\bar{\psi}_2)_{\infty}(-1) = (-1)^{[k]-[\ell]}$, then
\begin{enumerate}
	\item The series $D(s, f, g)$ can be meromorphically continued to the whole $s$-plane and it is holomorphic for $\Re(s)\geq 0$ if $k\neq \ell$ or $\Re(s)>0$ if $k=\ell$;
	\item The sum defining $D(s, f, g)$ is absolutely convergent for $\Re(s)>0$ if $g$ is a cusp form.
\end{enumerate}
\end{prop}

\begin{proof} \textbf{(1)} This follows from (14) and Lemma 17.2 (4) of \cite{Sh00} concerning the meromorphic continuation of the Eisenstein series $E_{k-\ell}(z, s)$.

\textbf{(2)} The operator $\sqrt{H_{\rho, R}} := \int_Y\rho(\sqrt{y})e^{-2\pi\tr(Ry)}|y|^{\frac{s+h}{2}}d^{\times}y$ is Hermitian and satisfies $\sqrt{H_{\rho, R}}\sqrt{H_{\rho, R}} = H_{\rho, R}$. We have
\[
	D(s, f, g) := \sum_{R\in\mathcal{S}}\nu_R \prec \sqrt{H_{\rho, R}}c_f(R, 1), \sqrt{H_{\rho, R}}c_g(R, 1)\succ
\]
and then by the Cauchy-Schwarz inequality
\[
	\left|\prec\sqrt{H_{\rho, R}}c_f(R, 1), \sqrt{H_{\rho, R}}c_g(R, 1)\succ\right|\leq\left[\norm{\sqrt{H_{\rho, R}}c_f(R, 1)}\norm{\sqrt{H_{\rho, R}}c_g(R, 1)}\right]^{\frac{1}{2}}
\]
where $\norm{\cdot}$ denotes the norm induced by $\prec\cdot, \cdot\succ$. From this we get
\[
	|D(s, f, g)|\leq[D(s, f, f)D(s, g, g)]^{\frac{1}{2}}.
\]
Therefore (ii) is given by showing convergence of $D(s, h, h)$ for $\Re(s)>0$ where $h$ is a cusp form. By (i) the series $D(s, h, h)$ is holomorphic for $\Re(s)>0$, is a Dirichlet series whose coefficients are non-negative, and so is convergent.
\end{proof}

For some applications later in this paper we want to be able to remove the assumption that $g$ is a cusp form in the second part of the above proposition. We can actually slightly modify the above proof to show the following,

\begin{prop} \label{dirichletconv_2} Let $k, \ell\in\frac{1}{2}\mathbb{Z}$, $f\in S_{\rho_k}(\Gamma_1, \psi_1)$, and $g\in M_{\rho_{\ell}}(\Gamma_2, \psi_2)$ be such that $(\psi_1\bar{\psi}_2)_{\infty}(-1) = (-1)^{[k]-[\ell]}$ and $k>\ell$. Assume that  $g(z) = \sum_{R\in S^+} c_g(R, 1)e^{2\pi i\tr(R\cdot z)} $ and that $g E_{k-\ell}(s_0+ \tfrac{n+1}{2})$ is a real analytic cusp form for some $s_0 \in \mathbb{N}$.  Then $D(s_0, f, g)$ is absolutely convergent.
\end{prop}

\begin{proof} The proof follows exactly as before after observing that we can still employ the identity
\[
D(s_0, g, g) = 2^{-1}N(\f{z})^{\frac{n(n+1)}{2}}\langle g, gE_{k-\ell}(s_0+ \tfrac{n+1}{2})\rangle \Vol(\Gamma\back\mathbb{H}_n)
\]
Indeed the fact that $g$ has expansion at infinity supported only at the full-rank matrices allow us to use the unfolding process as before, and of course use the fact that  the integral $\langle g, gE_{k-\ell}(s_0+ \tfrac{n+1}{2})\rangle$ is convergent thanks to the assumption on $gE_{k-\ell}(s_0+ \tfrac{n+1}{2})$ being a cusp form.
\end{proof}

\subsection{The function $D(s,f,\theta)$} For a fixed $\tau\in S_+$ such that $\prec c_f(\tau, 1), P(\sqrt{\tau}^{-1})\succ\ \neq 0$ we consider the theta series $\theta:= \theta_{\rho, \chi}(z) \in M_{\rho_{\ell}}(\Gamma',\psi_2)$ obtained in section 3. We recall that $\psi_2 = \chi \epsilon_{\tau}$ and $\ell = \frac{n}{2}$, and assume that $(\psi\chi)_{\infty}(-1) = (-1)^{[k]}$. 

We now consider $D(s,f,\theta)=\sum_{R \in \mathcal{S}}\nu_R \prec H_{\rho, R}^{n}(s)  c_f(R, 1), c_{\theta}(R, 1)\succ $ in which
\[
c_{\theta}(R, 1) = \sum_{\xi \in X_{R}} (\chi_{\infty}\chi^*)(\det(\xi)) P(\sqrt{\tau}\xi),
\]
and $X_{R} :=\{ \xi \in GL_n(\mathbb{Q}) \cap M_n(\mathbb{Z}) \,\,\,| \,\,\, R = \transpose{\xi} \tau \xi\}$. To give a more explicit description of the series $D(s, f, \theta)$ the value of the integral $H_{\rho}(s)P(1)$ is now calculated.

\begin{prop} \label{Gammacalc} The integral $H_{\rho}(s)P(1)$ has the following expression
\[ (4\pi)^{n(s+h+\lambda_P)} \rho(\xi^{-1}\sqrt{\tau}^{-1})H_{\rho}(s)P(1) =\bold{\Gamma}_{\rho}(s)P(\sqrt{\tau}^{-1}\transpose{\xi}^{-1})
\]
where $\lambda_P = \lambda_1+\cdots +\lambda_n$ is the weight of the vector $P(1)$ and 
\[
\bold{\Gamma}_{\rho}(s) :=  \pi^{n(n-1)/4} \left( \prod_{i=1}^n  \Gamma( s + h  + \lambda_i - \frac{i}{2} + \frac{1}{2}\right).
\]
\end{prop}

\begin{proof} By definition
\[
\rho(\xi^{-1} \sqrt{\tau}^{-1})H_{\rho}(s)P(1)=\int_Y  \rho(\xi^{-1} \sqrt{\tau}^{-1})P(y)  e^{-4 \pi\tr(y)}|y|^{s+h}d^{\times}y
\]
which latter integral we show to be $(4\pi)^{-n(s+h+\lambda_P)}\bold{\Gamma}_{\rho}(s) P( \sqrt{\tau}^{-1} \transpose{\xi^{-1}})$.

First we show that there is an $\alpha \in \mathbb{C}^{\times}$ such that 
\[
\int_{Y}  P(y)  e^{-4 \pi\tr(y)}|y|^{s+h}d^{\times}y = \alpha P(1).
\]

We write $V:=V_{\rho}$ for the representation space of $\rho$ and $W:=V_{\tau}$ for the representation space of $\tau$, the irreducible representation of $O(n)$ associated with $\rho$. Then we have the identifications $V \otimes W = V \otimes W^* = \text{Hom}(V,W) = M_{d,r}(\mathbb{C})$ where $d=\dim(V)$ and $r=\dim(W)$. In particular we have that the group $GL_ n\times O(n)$ acts on the set of pluriharmonic polynomials on $M_n$ with values on $M_{d,r}$ by $\mathbf{P}(g_1 x g_2) = \rho(\transpose{g}_2)\mathbf{P}(x) \tau(\transpose{g}_1)$. Notice that each such polynomial $\mathbf{P}$ consists of polynomials (columns) $P_j$, $j=1,\ldots,r$, that are pluriharmonic and $P_j(xg) = \rho(\transpose{g})P_j(x)$. In particular we may choose our polynomial $P$ above to be one of the columns of a polynomial $\mathbf{P}$. So it is enough to show that there is a constant $\alpha$ such that
 \[
\int_{Y}  \mathbf{P}(y)  e^{-4 \pi\tr(y)}|y|^{s+h}d^{\times}y = \alpha \mathbf{P}(1).
\]
We claim that we may pick the polynomial $\mathbf{P}$ such that $\mathbf{P}(1) = v \otimes w$ with $v$ a highest weight vector for $\rho$ and $w$ a highest weight for $\tau$, where here we use the identification above. Indeed given such a $\mathbf{P}$ we may find a matrix $A \in GL_n(\mathbb{C})$ such that $\mathbf{P}(A) \neq 0$ -- this is since $GL_n$ is dense in $M_n$ and the representation $\rho \otimes \tau$ is non trivial. That is, there exists a $\textbf{P}$ such that $\textbf{P}(1) \neq 0$. We now consider the set $S$ of all $\textbf{P}$ as above with the property $\textbf{P}(1) \neq 0$, and note that the space $R=\left\{ \textbf{P}(1) \in V \otimes W : \textbf{P} \in S \right\} \subseteq V \otimes W$ is invariant under the action of $GL_n \otimes O_n$. Indeed
\[
(g_1,g_2) \textbf{P}(1) = \textbf{P}(g_1 g_2) = \rho(\transpose{g_1}) \textbf{P}(1) \tau(\transpose{g}_2) \neq 0.  
\]
But the representation $\rho \otimes \tau$ is irreducible, and so $R$ must be equal to $V \otimes W$. That is, we can find a $\textbf{P}$ such that $\textbf{P}(1)$ is a highest weight vector. The proof of the proposition is now completed in the following two lemmas.
\end{proof}

\begin{lem} With notation as above there exists an $\alpha \in \mathbb{C}$ such that 
\[
\int_{Y}  \mathbf{P}(y)  e^{-4 \pi\tr(y)}|y|^{s+h}d^{\times}y = \alpha \mathbf{P}(1).
\]
\end{lem}

\begin{proof} We recall that every symmetric matrix $y$ may be written in the form $y = \transpose{a} \delta a$ (polar decomposition) with $a \in O(n)$ and $\delta = \diag[\delta_1, \ldots, \delta_n]$ a diagonal matrix.  Let $D := \{\diag[\delta_1, \dots, \delta_n]\mid \delta_i\in \mathbb{R}\}$, then

\begin{align*}
\int_{Y} & \mathbf{P}(y)  e^{-4 \pi\tr(y)}|y|^{s+h}d^{\times}y   \\
	&= c_0 \int_{O(n)} \int_{D} \mathbf{P}(\transpose{a} \delta a)  e^{-4 \pi\tr(\delta)}|\delta|^{s+h-\frac{n+1}{2}}  \left( \prod_{j < k} (\delta_k-\delta_j)\right) d\delta da\\
	&= c_0 \int_{O(n)} \rho(\transpose{a})  \int_{D} \mathbf{P}(\transpose{a} \delta)  e^{-4 \pi\tr(\delta)}|\delta|^{s+h-\frac{n+1}{2}} \left( \prod_{j < k} (\delta_k-\delta_j)\right)  d\delta da
\end{align*}

for some constant $c_0$. Since $\mathbf{P}(1) \in V \otimes W$ is a highest weight vector in the first component we know that $\mathbf{P}(\transpose{a}\delta) = [\rho(\delta)\mathbf{P}(1)]\tau(a) = \delta_1^{\alpha_1} \cdots \delta_n^{\alpha_n} \mathbf{P}(\transpose{a})$.

That is, the above integral reads
\begin{align*}
c_0 &\int_{O(n)} \rho(\transpose{a})  \int_{D} \mathbf{P}(\transpose{a})  \delta_1^{\alpha_1} \cdots \delta_n^{\alpha_n} e^{-4 \pi\tr(\delta)}|\delta|^{s+h-\frac{n+1}{2}}  \left( \prod_{j < k} (\delta_k-\delta_j)\right)  d\delta da\\
&=\mathbf{P}(1) c_0 \int_{O(n)}  \left(\int_{D}  \delta_1^{\alpha_1} \cdots \delta_n^{\alpha_n} e^{-4 \pi\tr(\delta)}|\delta|^{s+h-\frac{n+1}{2}}  \left( \prod_{j < k} (\delta_k-\delta_j)\right)  d \delta \right) da
\end{align*}
where of course we have used the fact that $\rho(\transpose{a}) \mathbf{P}(\transpose{a}) = \mathbf{P}(\transpose{a} a) = \mathbf{P}(1)$ since $a \in O(n)$.

\end{proof}

By the above lemma and the remark that our polynomial $P$ can be selected as a column polynomial of $\mathbf{P}$ as above we have established that 
\[
\int_{Y}  P(y)  e^{-4 \pi\tr(y)}|y|^{s+h}d^{\times}y = \alpha P(1).
\]

for some constant $\alpha \in \mathbb{C}$. We now calculate this constant. 

\begin{lem} We have that 
\[
	(4\pi)^{n(s+h+\lambda_P)}\alpha =\bold{\Gamma}_{\rho}(s) = \pi^{n(n-1)/4} \left( \prod_{i=1}^n  \Gamma( s + h  + \lambda_i - \frac{i}{2} + \frac{1}{2}\right). 
\]
\end{lem}
\begin{proof}
In principle we could try to calculate the above integral and the constant $c_0$, however we can do it in a different way. Instead we calculate
\begin{align*}
\prec \alpha P(1), P(1) \succ\ &=\ \prec \int_{Y}  P(y)  e^{-4 \pi\tr(y)}|y|^{s+h}d^{\times}y, P(1) \succ \\
	&=\int_{Y}  \prec P(y),P(1) \succ  e^{-4 \pi\tr(y)}|y|^{s+h}d^{\times}y, 
\end{align*}
Our method is similar to the one used in \cite[page 88]{BP2}. We first use Gauss decomposition and write $y = \transpose{T} T$ where $T$ is lower triangular. Then by $P(\transpose{T}T) = \rho(\transpose{T})P(\transpose{T})$ and the fact that $\prec \rho(\transpose{T})\cdot, \cdot\succ\ =\ \prec\cdot, \rho(T)\cdot\succ$ we get
\begin{align*}
	\prec P(\transpose{T}T), P(1)\succ\ =\ \prec P(\transpose{T}), \rho(T)P(1)\succ\ =\ \prec P(\transpose{T}), P(\transpose{T})\succ
\end{align*}
and so the integral is equal to 
\begin{align*}
\int_{Y}&  \prec P(\transpose{T} T),P(1) \succ  e^{-4 \pi\tr(y)}|y|^{s+h}d^{\times}y \\
	&= \int_{Y}  \prec  P(\transpose{T}),P(\transpose{T}) \succ  e^{-4 \pi\tr(y)}|y|^{s+h}d^{\times}y.
\end{align*}

But we have $P(\transpose{T}) = t_{11}^{\lambda_1} \ldots t_{nn}^{\lambda_{n}} P(1)$ since $\transpose{T}$ is upper triangular and $P$ is a highest weight vector. In particular we compute that
\[
\prec  P(1),P(1) \succ \int_{Y}   \prod_{i=1}^n t_{ii}^{2 \lambda_i} e^{-4 \pi\tr(y)}|y|^{s+h}d^{\times}y, \,\,\,\,y = \transpose{T} T.
\]
The last integral has been computed by Maass in \cite[pp. 76--80]{Maass} and is equal to 
\[
	(4\pi)^{-n(s+h+\lambda_P)}\pi^{n(n-1)/4}  \prod_{i=1}^n  \Gamma\left( s + h + k + \lambda_i - \tfrac{i}{2} + \tfrac{1}{2}\right).
\]
Hence so is $\alpha$.

\end{proof}

Recalling the definition of $D(s, f, \theta)$ and $c_{\theta}(R, 1)$ we have
\begin{align*}
D(s,f,\theta) &= \sum_{R \in \mathcal{S}}\nu_R \prec H_{\rho,R}(s)  c_f(R, 1), \sum_{\xi \in X_{R}} (\chi_{\infty}\chi^*)(|\xi|) P(\sqrt{\tau}\xi)\succ\ \\
&= \sum_{R \in \mathcal{S}}\nu_R \prec c_f(R, 1), \sum_{\xi \in X_{R}} (\chi_{\infty}\chi^*)(|\xi|) H_{\rho,R}(s)  P(\sqrt{\tau}\xi)\succ ,
\end{align*}
where we have used the fact that $H_{\rho,R}(s)$ is hermitian.  In the inner sum we may write $R = \transpose{\xi} \tau \xi$ so that, by (13), the summation above is equal to
\begin{align*}
 &\sum_{R \in \mathcal{S}}\nu_R\sum_{\xi \in X_{R}} \chi(\det(\xi))\prec c_f(\transpose{\xi} \tau \xi, 1),H_{\rho,\transpose{\xi} \tau \xi}(s)  P(\sqrt{\tau} \xi)\succ\ \\
& = \sum_{\xi \in X} (\chi_{\infty}\chi^*)(\det(\xi))  \prec c_f(\transpose{\xi} \tau \xi, 1), \rho(\xi^{-1} \sqrt{\tau}^{-1} H_{\rho}(s) P(1)\succ \det(\transpose{\xi} \tau \xi)^{-(s+h)}
\end{align*}
where $X = (GL_n(\mathbb{Q}) \cap M_n(\mathbb{Z})) /GL_n(\mathbb{Z})$.

By Proposition \ref{Gammacalc} we have
\[
(4\pi)^{n(s+h+\lambda_P)}D(s,f,\theta) = \bold{\Gamma}_{\rho}(s) \sum_{\xi \in X} \chi(|\xi|)  \prec c_f(\transpose{\xi} \tau \xi, 1), P(\sqrt{\tau}\transpose{\xi}^{-1}) \succ |\transpose{\xi} \tau \xi|^{-(s+h)}.
\]

\section{Andrianov-Kalinin Identity for Vector Valued Siegel Modular Forms}

In this section we introduce the $L$-function attached to a cuspidal eigenform $f$, and relate it to the Dirichlet series $D(s, f,\theta)$ studied in the previous sections. We closely follow Shimura's method in the scalar weight case as for example is done in \cite{Sh94,Sh00}. Using this relation we then obtain the main results of the paper. We remind the reader that all the theorems below are subject to the assumption that the representation $\rho$ is in $\tau(\Sigma)$. \newline

We define
\begin{align*}
	B:&= \prod_p M_n(\mathbb{Z}_p)\cap GL_n(\mathbb{Q}_p) \\
	E:&= \prod_p GL_n(\mathbb{Z}_p).
\end{align*}
If $e\in B$ and $\sigma = \diag[\tilde{e}, e]$ then with the finite decomposition $\Gamma\sigma\Gamma = \bigsqcup_{\gamma\in C}\Gamma\gamma$ we define the action of $T_{e, \psi}$ on $f\in\mathcal{M}_{\rho_k}(\Gamma, \psi)$ by
\[
	f|T_{e, \psi} := \sum_{\gamma\in C}\psi_{\f{c}}(|a_{\gamma}|)^{-1}f|_{\rho_k}\gamma.
\]
Adelically this is given by the decomposition $D\sigma D = \bigsqcup_y Dy$ for $D := D[\f{b}^{-1}, \f{bc}]$ and $y\in G_{\mathbf{h}}$ and then
\[
	(f|T_{e, \psi})_{\mathbb{A}}(x) := \sum_y \psi_{\f{c}}(|a_y|)^{-1}f_{\mathbb{A}}(xy^{-1})
\]
with $x\in G_{\mathbb{A}}$ or $M_{\mathbb{A}}$ depending on the parity of $2k$. For a positive integer $n$ let
\[
	T_{\psi}(n) := \sum_{e\in E\backslash B/E, |e| = n}T_{e, \psi}
\]
and we assume that $f$ is an eigenform so that $f|T_{\psi}(n) = \lambda(n)f$ for $\lambda(n)\in\mathbb{C}$. Let $\psi' = \psi/\psi_{\f{c}}$ and for any Hecke character $\chi$ such that $(\psi\chi)_{\infty}(x) = \text{sgn}(x)^{[k]}$ define the operator
\[
	\f{T}_{\psi, \chi} := \sum_{e\in E\backslash B/E}T_{e, \psi}\psi'(|e|)\chi^*(|e|)|e|^{-s} = \sum_{n=1}^{\infty}T(n)\psi'(n)\chi^*(n)n^{-s}
\]
where $\chi^*$ is the ideal Hecke character associated to $\chi$. For such an eigenform one defines the standard $L$-function as follows. For any prime $p$ we can associate to $f$ the Satake $p$-parameters $\lambda_{p, i}$ where $i\in\{1, \dots, n\}$, as per \cite[p. 554]{Sh94}. If $p\nmid\f{c}$ then define
\[
	L_p(t) :=
		\begin{cases}
			\displaystyle (1-p^nt)\prod_{i=1}^n(1-p^n\lambda_{p, i}t)(1-p^n\lambda_{p, i}^{-1}t) &\text{if $k\in\mathbb{Z}$} \\
			\displaystyle \prod_{i=1}^n (1-p^n\lambda_{p, i}t)(1-p^n\lambda_{p, i}^{-1}t) &\text{if $k\notin\mathbb{Z}$}
		\end{cases}
\]
and if $p\mid\f{c}$ then $L_p(t) = \prod_{i=1}^n (1-p^n\lambda_{p, i}t)$ in either case. For a complex variable $s$ the standard $L$-function is subsequently given by
\[
	L_{\psi}(s, f, \chi) = \prod_pL_p(\psi'(p)\chi^*(p)p^{-s})^{-1}.
\]
We remark here that for $f$ a cusp form we have that $L_{\psi}(s, f, \chi)$ is absolute convergent for $\Re(s) > 2n+1$ if $k$ is an integer and for $\Re(s) > 2n$ if $k$ is half-integer. This is shown, for example, in \cite[Lemma 20.12]{Sh00} in the scalar weight situation, but the same argument carries to the vector valued situation. \newline

By \cite[Lemma 4 and the discussion after]{W} we can find $\tau\in S_+$ such that
\[
	\prec c_f(\tau, 1), P(\sqrt{\tau}^{-1})\succ\ \neq 0
\]
and then define a Dirichlet series
\[
	D_{\tau}(s, f, \chi) := \sum_{\xi\in B/E}(\psi\chi^*)(|\xi|)\prec c_f(\tau, \xi), P(\sqrt{\tau}^{-1})\succ |\xi|^{-s}\|\xi\|_{\mathbb{A}}^{-n-1}.
\]
Much like Theorem 5.1 and Corollary 5.2 of \cite{Sh94} and \cite{Sh95} we are able to obtain a relation between $D_{\tau}(s, f, \chi)$ and $L_{\psi}(s, f, \chi)$. Let $c_{\f{T}}(\tau, b) := c(\tau, b;f|\f{T}_{\psi, \chi})$ and immediately we know on the one hand that
\begin{align}
	c_{\f{T}}(\tau, b) = \left(\sum_{n=1}^{\infty}\lambda(n)\psi'(n)\chi^*(n)n^{-s}\right)c_f(\tau, b).
\end{align}
On the other hand we use the definition of the Hecke operators and the coset decompositions given in Lemma 2.6 of \cite{Sh94} to obtain an alternate expression for $c_{\f{T}}(\tau, b)$. This lemma tells us that we can take as our coset representatives
\[
	y = \begin{pmatrix} g^{-1}h & g^{-1}\sigma\tilde{h} \\ 0 & \transpose{g}\tilde{h}\end{pmatrix}
\]
for suitable $\sigma\in S$; $g, h\in B$. Using the adelic Hecke action on $f_{\mathbb{A}}$ and mimicking p.554 of \cite{Sh94} we obtain
\[
	c_{\f{T}}(\tau, b) = \sum_{g, h}(\psi\chi^*)(|h^{-1}g|)c_f(\tau, bh^{-1}g)|gh|^{-s}\|g\|_{\mathbb{A}}^{-n-1}\alpha_{\f{c}}(B_k\tilde{h}\transpose{b}\tau bh^{-1})
\]
where $\alpha_{\f{c}} = \prod_{p\nmid\f{c}}\alpha_p$ and $\alpha_p$ is defined by \cite[(2.5b)]{Sh94} if $k\in\mathbb{Z}$ and \cite[(4.2)]{Sh95} if $k\notin\mathbb{Z}$. The rest of the proof of Theorem 5.1 in \cite{Sh94} and \cite{Sh95} now follows and, noting that $\psi$ is trivial on global ideles, this gives
\[
	c_{\f{T}}(\tau, b) = \alpha_{\f{c}}(B_k\tau)\sum_{\xi\in B/E}(\psi\chi^*)(|\xi|)c_f(\tau, b\xi)|\xi|^{-s}\|\xi\|_{\mathbb{A}}^{-n-1}
\]
and in particular
\begin{align}
	\prec c_{\f{T}}(\tau, 1), P(\sqrt{\tau}^{-1})\succ\ = \alpha_{\f{c}}(B_k\tau)D_{\tau}(s, f, \chi).
\end{align}

Using \cite[p. 554 and equation (5.8)]{Sh94} when $k\in\mathbb{Z}$ and \cite[(5.4a--b), (5.5)]{Sh95} when $k\notin\mathbb{Z}$ we have
\begin{align}
	\begin{split}
		\Lambda_{\f{c}}^{2n, k}(\tfrac{s}{2}, \chi\psi)\sum_{n=1}^{\infty}\lambda(n)\psi'(n)\chi^*(n)n^{-s} &= L_{\psi}(s, f, \chi) \\
		\Lambda_{\f{c}}^{n, k-\frac{n}{2}}(\tfrac{s}{2}, \chi\psi\epsilon_{\tau})\alpha_{\f{c}}(B_k\tau) &= \prod_{p\in\mathbf{b}}g_p(\chi'(p)\chi^*(p)p^{-s})
	\end{split}
\end{align}
where $B = N(\f{b})$, $B_k = B^{2k-2[k]-1}$; $\mathbf{b}$ is the set of all primes $p$ that divide either the numerator or denominator of $2^{-n\pmod{2}}|2B_k\tau|$; $g_p\in\mathbb{Q}[t]$ such that $g_p(0) = 1$; and $\Lambda$ is a product of Dirichlet $L$-functions defined, for $1\leq m\in\mathbb{Z}$, $\kappa\in\frac{1}{2}\mathbb{Z}$, character $\eta$, and integral ideal $\f{x}$, by
\begin{align*}
\Lambda_{\f{x}}^{m, \kappa}(s, \eta) := \begin{cases}\displaystyle L_{\f{x}}(2s, \eta)\prod_{i=1}^{[m/2]}L_{\f{x}}(4s-2i, \eta^2) &\text{if $\kappa = [\kappa]$} \\
	\displaystyle \prod_{i=1}^{[(m+1)/2]}L_{\f{x}}(4s-2i+1, \eta^2) &\text{if $\kappa\neq[\kappa]$}.
\end{cases}
\end{align*}

Combining (15), (16), and (17) yields
\begin{align}
	\begin{split}
	\prec c_f(\tau, 1), P(\sqrt{\tau}^{-1})\succ &L_{\psi}(s, f, \chi) \\
	&= \prod_{p\in\mathbf{b}}g_p(\psi'(p)\chi^*(p)p^{-s})\Lambda_{\f{c}}\left(\tfrac{2s-n}{4}\right)D_{\tau}(s, f, \chi)
	\end{split}
\end{align}
where $\Lambda_{\f{c}}(s) = \Lambda_{\f{c}}^{n, k-\frac{n}{2}}(s, \chi\psi\epsilon_{\tau})$.

For relating our two Dirichlet series $D_{\tau}(s, f, \chi)$ and $D(s, f, \theta)$ we need to turn our adelic series $D_{\tau}(s, f, \chi)$ into a global one. By the strong approximation theorem we have $GL_n(\mathbb{Q})_{\mathbb{A}} = GL_n(\mathbb{Q})\times (GL_n(\mathbb{R})\times E)$ and as in (5.16) of \cite{Sh94} we have
\[
	D_{\tau}(s, f, \chi) = \sum_{\xi\in X} (\psi_{\mathbf{h}}\chi^*)(|\xi|)\prec c_f(\tau, \xi), P(\sqrt{\tau}^{-1})\succ |\xi|^{n+1-s}.
\]

Using Theorem \ref{adelicexp} (3) we get
\begin{align*}
	D_{\tau}(s, f, \chi) &= \sum_{\xi\in X}(\psi_{\mathbf{h}}\chi^*)(|\xi|)\prec \rho_k(\tilde{\xi})c_f(\transpose{\xi}\tau\xi, 1), P(\sqrt{\tau}^{-1})\succ|\xi|^{n+1-s} \\
	&= \sum_{\xi\in X}(\psi_{\mathbf{h}}\chi^*)(|\xi|)\prec c_f(\transpose{\xi}\tau\xi, 1), P(\sqrt{\tau}^{-1}\tilde{\xi})\succ|\xi|^{n+1-k-s} \\
	& = \sum_{\xi\in X}(\chi_{\infty}\chi^*)(|\xi|)\prec c_f(\transpose{\xi}\tau\xi, 1), P\sqrt{\tau}^{-1}\tilde{\xi})\succ |\xi|^{n+1-k-s}
\end{align*}
where in the last line we used the fact that $(\psi\chi)_{\infty}(|\xi|) = \text{sgn}(|\xi|)^{[k]} = 1$ since $\xi$ is taken modulo $GL_n(\mathbb{Z})$, and the fact that $\psi(|\xi|) = 1$ as $\xi$ is global.

By the previous section we have
\begin{align*}
	\begin{split}
(4\pi)^{n\lambda_P}&|4\pi\tau|^{s+\frac{2k+n}{4}}D(s, f, \theta) = \\
	&\bold{\Gamma}_{\rho}(s)\sum_{\xi\in X}(\chi_{\infty}\chi^*)(|\xi|)\prec c_f(\transpose{\xi}\tau\xi, 1), P(\sqrt{\tau}^{-1}\tilde{\xi})\succ|\xi|^{-2s-k-\frac{n}{2}}
	\end{split}
\end{align*}
and so we get
\begin{align}
	\bold{\Gamma}_{\rho}(s')D_{\tau}(s, f, \chi) = (4\pi)^{n\lambda_P}|4\pi\tau|^{s'+\frac{2k+n}{4}}D(s', f, \theta)
\end{align}
where $s' = \frac{2s-3n-2}{4}$. \newline

 Assume now that $k\geq \frac{n}{2}$ ($k>\frac{n}{2}$ if $k-\frac{n}{2}\notin\mathbb{Z}$), let $\bar{n} = n\pmod{2}\in\{0, 1\}$, and define 
\begin{align*}
	\Gamma^{k, n}(s) := \begin{cases}
		\displaystyle\Gamma_n\left(s+\frac{k-n}{2}\right) &\text{if $n<k\notin\mathbb{Z}$} \\
		\displaystyle\Gamma\left(s+\frac{k-n-\bar{n}}{2} - \left[\frac{k-\bar{n}}{2}\right]\right)\Gamma_n\left(s+\frac{k-n}{2}\right) &\text{if $n<k\in\mathbb{Z}$}. \\
		\displaystyle\Gamma_{2k-n+1}\left(s+\frac{k-n}{2}\right)\displaystyle\prod_{i= k-\frac{n}{2}+1}^{\left[n/2\right]}\Gamma\left(2s-\frac{n}{2}-i\right) &\text{if $\frac{n}{2}\geq k-\frac{n}{2}\in\mathbb{Z}$} \\
		\displaystyle\Gamma_{2k-n+1}\left(s+\frac{k-n}{2}\right)\displaystyle\prod_{i=[k-\frac{n}{2}]+1}^{\left[(n-1)/2\right]}\Gamma\left(2s-\frac{n+1}{2}-i\right) &\text{if $\frac{n}{2}\geq k-\frac{n}{2}\notin\mathbb{Z}$}.
	\end{cases}
\end{align*}

Combining equations (14), (18), and (19) then gives us the final integral expression for $L_{\psi}(s, f, \chi)$ which we give in a theorem below.

\begin{thm}\label{integralexp} Let $f\in \mathcal{M}_{\rho_k}(\Gamma, \psi)$ be a non-zero Hecke eigenform where $\Gamma = \Gamma[\mathfrak{b}^{-1}, \mathfrak{bc}]$ for a fractional ideal $\mathfrak{b}$ and integral ideal $\mathfrak{c}$ of $\mathbb{Q}$ ($\Gamma\leq\f{M}$ if $k\notin\mathbb{Z}$), and $\psi$ is a Hecke character. Select a $\tau\in S_+$ so that
\[
	\prec c_f(\tau, 1), P(\sqrt{\tau}^{-1})\succ \neq 0
\]
and fix this $\tau$. Let $\chi$ be another Hecke character such that $(\psi\chi)_{\infty} = \sgn(x)^{[k]}$, and let $\chi^*$ denote the corresponding ideal Hecke character. Then we have
\begin{align*}
L_{\psi}(s, f, \chi)\bold{\Gamma}_{\rho}(s')\Gamma^{k, n}\left(\tfrac{s}{2}\right) &= [2\prec c_f(\tau, 1), P(\sqrt{\tau}^{-1})\succ]^{-1}N(\f{z})^{\frac{n(n+1)}{2}}\Vol(\Gamma'\back\mathbb{H})\\
	&\times(4\pi)^{n\lambda_P}|4\pi\tau|^{s'+\frac{2k+n}{4}} \prod_{p\in\mathbf{b}}g_p(\psi'(p)\chi^*(p)p^{-s}) \\
	&\times\left(\frac{\Lambda_{\mathfrak{c}}}{\Lambda_{\mathfrak{y}}}\right)\left(\tfrac{2s-n}{4}\right)\langle f, \theta\mathcal{E}(z, \tfrac{2s-n}{4})\rangle
\end{align*}
where $s' = \frac{2s-3n-2}{4}$; $\bold{\Gamma}_{\rho}$ is defined in the previous section; $\theta$ is defined as in Section 3 with weight $\ell = \frac{n}{2}$, level $(\mathfrak{b}', \mathfrak{c}')$, and character $\chi\epsilon_{\tau}$;
\[
	\mathcal{E}(z, s) := \overline{\Gamma^{k, n}(s+\tfrac{n}{4})\Lambda_{\mathfrak{y}}(s)}E_{k-\ell}(z, s; \bar{\psi}\chi\epsilon_{\tau}, \Gamma')
\]
where $\mathfrak{z} := \f{b}+\f{b}'$, $\mathfrak{y} = \f{z}^{-1}(\f{b}\mathfrak{c}\cap\f{b}'\mathfrak{c}'$), and $\Gamma' = \Gamma[\mathfrak{z}^{-1}, \mathfrak{zy}]$; $\mathbf{b}$ is a finite set of primes; $g_p\in\mathbb{Q}[t]$ and $\Lambda_{\mathfrak{x}}(s) = \Lambda_{\f{x}}^{n, k-\frac{n}{2}}(s, \chi\psi\epsilon_{\tau})$ are given above.
\end{thm}

Let $Z_{\psi}(s, f, \chi) := \bold{\Gamma}_{\rho}\left(  \frac{2s-3n-2}{4} \right) \Gamma^{k, n}\left(\tfrac{s}{2}\right)L_{\psi}(s, f, \chi)$, then the integral expression of theorem \ref{integralexp} easily allows us to determine the analyticity and location of any poles, since these can only occur at the possible poles of $\mathcal{E}(z, \frac{2s-n}{4})$ and the factor $\left(\frac{\Lambda_{\mathfrak{c}}}{\Lambda_{\mathfrak{y}}}\right)\left(\tfrac{2s-n}{4}\right)$. Note that the latter is just a finite product of Euler factors.

\begin{thm}\label{analyticity} If $f\in\mathcal{M}_{\rho_k}(\Gamma, \psi)$ is a non-zero Hecke eigenform and $\chi$ is a character such that $(\psi\chi)_{\infty}(-1) = (-1)^{[k]}$, then the function $Z_{\psi}(s, f, \chi)$ has only finitely many poles all of which are simple. 
\begin{itemize}
\item If $(\psi\chi)^2 \neq 1$ then $Z_{\psi}(s, f, \chi)$ may have simple poles at  the values of $s$ where the factor $\frac{\Lambda_{\mathfrak{c}}}{\Lambda_{\mathfrak{y}}}\left(\tfrac{2s-n}{4}\right)$ has poles.

\item If $(\psi\chi)^2 = 1$ and $\mathfrak{y}\neq\mathbb{Z}$  then in addition to the possible poles by the factor  $\frac{\Lambda_{\mathfrak{c}}}{\Lambda_{\mathfrak{y}}}\left(\tfrac{2s-n}{4}\right)$, there may be some poles occurring only in the following sets.
	\begin{enumerate}
		\item If $k>n$ then $Z_{\psi}(s, f, \chi)$ has a single pole at $s = n+1$ only if $k\in\mathbb{Z}$ and $k-n\in 2\mathbb{Z}$;
		\item If $\frac{n}{2}\leq k\leq n$ then the possible poles of $Z_{\psi}(s, f, \chi)$ occur only in the sets
\[
		\begin{cases}
			\{ j \mid j\in\mathbb{Z}, n+1\leq j\leq 2n+1-k\} &\text{if $k-\frac{n}{2}\in\mathbb{Z}$} \\
			\{ j+\frac{1}{2}\mid j\in\mathbb{Z}, n+1\leq j\leq 2n+\tfrac{1}{2}-k\} &\text{if $k-\frac{n}{2}\notin\mathbb{Z}$};
		\end{cases}
\]
	\end{enumerate}
If, on the other hand, we have $(\psi\chi)^2 = 1$, $\mathfrak{y} = \mathbb{Z}$, and $k-\frac{n}{2}\in\mathbb{Z}$ then in addition to the potential poles specified in the first set of (2) there may also be poles in
\[
	\{j\in\mathbb{Z}\mid \left[\tfrac{n+1}{2}\right]\leq j\leq n\}.
\]
\end{itemize}
\end{thm}

\begin{proof} This follows by determining the poles of the Eisenstein series which has been done in Theorem 7.3 of \cite{Sh94}.
\end{proof}

\begin{rem} \label{remark} We now make the following remarks.

\begin{enumerate}
\item We note here that the factor $\left(\frac{\Lambda_{\mathfrak{c}}}{\Lambda_{\mathfrak{y}}}\right)\left(\tfrac{2s-n}{4}\right)$ also appears in the scalar weight situation in \cite{Sh94}. Actually in that paper Shimura gives some conditions \cite[Proposition 8.3]{Sh94} such that this factor is trivial, which can be also used here. We refer to \cite{Sh94} for this. 

\item We moreover remark that the location of poles, and their order, of $Z_{\psi}(s, f, \chi)$ can be studied by using the doubling method, as has been done, for example by Shimura in \cite{Sh95} for the scalar weight case and by Piatetski-Shapiro and Rallis \cite{PR1} for the general vector valued case. However as we have remarked already in the introduction the result of the two methods (Rankin-Selberg vs. Doubling method) already in the scalar weight situation do not overlap (see the discussion in \cite[Remark 6.3]{Sh95}). We further mention that even though the factor  $\left(\frac{\Lambda_{\mathfrak{c}}}{\Lambda_{\mathfrak{y}}}\right)\left(\tfrac{2s-n}{4}\right)$ does not appear in the doubling method, it seems that one has to make other assumption regarding the behavior of $f$ at the primes dividing the level its level (see for example the conditions on Theorem 6.1 in \cite{Sh95}).

\item We mention that the theorem above is not covered by the result in \cite{PR2}. Indeed here we include all the Euler factors, we twist by characters and we make the gamma factors precise. Note for example that by making these factors precise allow one to locate also possible (trivial) zeros of  $L_{\psi}(s, f, \chi)$ forced by the gamma factors. 
\end{enumerate}
\end{rem}

We now establish some non-vanishing results for $L_{\psi}(s, f, \chi)$  beyond the range of absolute convergence. 
Fix an odd non-trivial character $\chi_0$ with conductor $p\neq 2$. Furthermore fix subsequent choices of diagonal $\tau\in S_+$ and $Q\in GL_n(\mathbb{Q}_{\mathbf{h}})$ for which, by Theorem \ref{theta}, $\theta_{\rho, \chi_0}(z)$ is a cusp form. Then we say that $f$ is $\chi_0$-ordinary if 
\begin{enumerate}
\item $\prec c_f(\tau, 1), P(\sqrt{\tau}^{-1})\succ\ \neq 0$. 
\item $
 (\psi \chi_0)_{\infty}(-1) = (-1)^{[k]}.$
\end{enumerate}
Note that since $\chi_0$ is odd, the condition (2) is only satisfied for $n$ even. By Theorem \ref{adelicexp} (3) and (4) the condition that $\tau$ be diagonal is non-exacting. 

\begin{thm}\label{Lconv} Let $\chi$ be any character with $(\psi \chi)_{\infty}(-1) = (-1)^{[k]}$. Then the function $L_{\psi}(s, f, \chi)$ can be meromorphically continued to the whole $s$-plane. Furthermore if $n$ is even and if, for an odd character $\chi_0$ of conductor $p\neq 2$, we have that $f$ is $\chi_0$-ordinary then the $L$-function obtained by removing the Euler factor at $p$,
\[
	L_{\psi}^{(p)}(s, f, \chi) := L_{\psi}(s, f, \chi)L_p(\psi'(p)\chi^*(p)p^{-s}),
\]
is convergent, and hence non-zero, for $\Re(s)>\frac{3n}{2}+1$.
\end{thm}

\begin{proof} Meromorphic continuation is given by the integral expression of Theorem \ref{integralexp} and continuation of the Eisenstein series there.

Consider equation (18) with $\chi = \chi_0$ which relates the Euler product of $L_{\psi}(s, f, \chi_0)$ with the Dirichlet series $D_{\tau}(s, f, \chi_0)$. Note that the product $\prod_{q\in\mathbf{b}}g_q(\psi'(q)\chi_0^*(q)q^{-s})$ is just finite and since, by assumption, $\Re(\frac{2s-n}{4})\geq 1$ so is $\Lambda_{\f{y}}(\frac{2s-n}{4})$. Therefore by Lemma 22.7 of \cite{Sh00} the convergence and the non-vanishing of $L_{\psi}(s, f, \chi_0)$ rests on the convergence of $D_{\tau}(s, f, \chi_0)$ which, in turn, rests on the convergence of $D(\frac{2s-3n-2}{4}, f, \theta_{\rho, \chi_0})$ by the relation of (19). Since $\theta_{\rho, \chi_0}$ is a cusp form by Theorem \ref{theta}, then by Proposition \ref{dirichletconv} (2) the series $D(\frac{2s-3n-2}{4}, f, \theta_{\rho, \chi_0})$ is convergent for $\Re(s)>\frac{3n}{2}+1$. Hence the convergence and non-vanishing of $L_{\psi}(s, f, \chi_0)$ has been established.

Now let $\chi$ be any character, and remove the Euler factor of $p$ from $L_{\psi}(s, f, \chi)$ to get $L_{\psi}^{(p)}(s, f, \chi)$. The Euler products of both $L_{\psi}^{(p)}(s, f, \chi)$ and $L_{\psi}(s, f, \chi_0)$ are over all primes $q\neq p$, and so the Euler product of $L_{\psi}^{(p)}(s, f, \chi)$ is just that of $L_{\psi}(s, f, \chi_0)$ twisted by the $\mathbb{T}$-valued character $\chi\chi_0^{-1}$. Lemma 22.7 of \cite{Sh00} then tells us that such an Euler product is convergent and non-vanishing for $\Re(s)>\frac{3n}{2}+1$.
\end{proof}

The fact that the character $\chi_0$ has to be odd put some limitations on the scalar weight $k$ as can be seen by condition (2) in the definition of $\chi_0$-ordinary. However if we assume that $n$ is divisible by 8, $k\in\mathbb{Z}$, $\Gamma = \Gamma[p,p]$  for some prime $p$, and the conductor of $\psi$ is equal to $p$, then we can lift the assumption that $\chi_0$ has to be odd. 

Let us take $ 8 | n$ and we may assume that we can choose $\tau$ and $r$ such that for a character $\chi$ of conductor $p$ the theta series $\theta_{\rho,\chi}$ is of level $\Gamma[p,p]$. Let
\[
	g(z,s):= \theta_{\rho,\chi} E_{k-\frac{n}{2}}(z, \tfrac{2s-n}{4}; \bar{\psi}\chi, \Gamma')
\]
where $\Gamma' := \Gamma[1, p] = \Gamma_0(p)$ and note $\bar{\psi}\chi$ has conductor $p$. Our aim is to determine $s_0 \in \mathbb{Z}$ such that $g(z,s_0)$ is a real analytic cusp form. Before we go any further we remark that $E_{k-\frac{n}{2}}(z, \frac{2s-n}{4}; \bar{\psi}\chi, \Gamma')= E_{k-\frac{n}{2}}(z, \frac{2s-n}{4}; \bar{\psi}\chi, \Gamma)$. Indeed, by Proposition 18.14 in \cite{Sh97}, they have the same Fourier expansion at the zero cusp. 

\begin{prop} \label{cuspidality_product}Let $s_0 \in \mathbb{Z}$ satisfy $\frac{3n}{2}+1 < s_0 \leq k$ and $s_0-k \in 2 \mathbb{Z}$. Then $g(z,s_0)$ is a real analytic cusp form.

\end{prop}
\begin{proof}
For the values of $s_0$ as in the statement the Eisenstein series is absolutely convergent. We set $\mu := \frac{2s_0-n}{2}$ and observe by Theorem 17.9 in \cite{Sh00} that $E_{k-\frac{n}{2}}(z, \frac{\mu}{2}; \bar{\psi}\chi, \Gamma')$ is a real analytic Siegel modular form for $n+1 < \mu < k - \frac{n}{2}$. We will now explore the behavior of this series at the cusps. Notice that since the Eisenstein series is with respect to the group $\Gamma_0(p)$ we have to consider only the cusp at infinity and the cusp at zero. We claim that the Eisenstein series is cuspidal at the zero cusp, i.e. that it does not have any Fourier coefficient at singular symmetric matrices. Following the notation in \cite{Sh00} we set $E_{k-\frac{n}{2}}^*(z, \frac{\mu}{2}) := E_{k-\frac{n}{2}}(z, \frac{\mu}{2}; \bar{\psi}\chi, \Gamma') |_k \eta$. We first consider the case that $\mu = k-\frac{n}{2}$. In this case we know by \cite[17.6]{Sh00} that the series $E_{k-\frac{n}{2}}^*(z, \frac{\mu}{2})$ has Fourier expansion supported only on positive definite matrices. We can now extend this to the rest of the values of $\mu$. We set $p := (k - \frac{n}{2} - \mu)/2$ and write $\Delta_{\mu}^{p}$ for the differential operator defined in \cite[page 146]{Sh00}, where it is also shown that 
\[
\Delta_{\mu}^{p} \left(E_{\mu}^*(z, \tfrac{\mu}{2})\right) = E_{k-\frac{n}{2}}^*(z, \tfrac{\mu}{2})
\] 
Since $E_{\mu}^*(z, \frac{\mu}{2})$ has Fourier expansion supported only on positive definite matrices, and the differential operators preserve this property, (see equation ($*$) in \cite[proof of Theorem 14.12]{Sh00}), we have that also $E_{k-\frac{n}{2}}^*(z, \frac{\mu}{2})$ has a Fourier expansion supported on positive definite matrices for all $\mu$ with $n+1 < \mu < k - \frac{n}{2}$. That is, we have now established that for values of $s_0 \in \mathbb{Z}$ as in the proposition the series $E_{k-\frac{n}{2}}(z, \frac{2s_0-n}{4}; \bar{\psi}\chi, \Gamma')$ is a real analytic modular form, and its Fourier expansion at the zero cusp is supported at positive definite matrices. We now turn to $g(z,s_0)$.

Thanks to the assumption that the congruence subgroup is $\Gamma[p,p]$ we are dealing with two kinds of cusps, namely $m(s)$ and $m(s) \eta$, in the notation we used in section 3. We first note that the Eisenstein series is invariant under $m(s)$. Furthermore, as we have seen in section 3 the theta series is cuspidal at the cusps $m(s)$ thanks to the non-trivial character $\chi$. So the product is also cuspidal since we have established that the Eisenstein series does not have any negative definite support in its Fourier expansion. On the other hand for the second kind of cusps, $m(s) \eta$, we have shown that the Eisenstein series has no singular terms, and hence the same holds for the product. 
\end{proof}

We can now prove the following version of Theorem \ref{Lconv},

\begin{thm}\label{Lconv2} 
	With notation as before,  we assume $8 | n$, and assume there exists of $\tau\in S_+$ such that $\prec c_f(\tau, 1), P(\sqrt{\tau}^{-1})\succ\ \neq 0$. Moreover assume that $\mathfrak{b} = p^{-1}\mathbb{Z}$ and $\mathfrak{c} = p^2\mathbb{Z} $ for some prime $p$. Then, for any Dirichlet character $\chi$ the Euler product of
\[
	L_{\psi}^{(p)}(s_0, f, \chi) := L_{\psi}(s_0, f, \chi)L_p(\psi'(p)\chi^*(p)p^{-s_0})
\]
is non-zero, for $s_0 \in\mathbb{Z}$ with $\frac{3n}{2}+1 < s_0 \leq k$, and $s_0-k \in 2 \mathbb{Z}$.
\end{thm}

\begin{proof} The proof of this theorem is almost identical to the previous one, but now we are not bounded to make the theta series cuspidal. That is, we choose a $\chi_0$, not necessarily odd, such that $(\psi \chi_0)_{\infty}(-1) = (-1)^{k}$ and a $Q \in GL_n(\mathbb{Q}_{\mathbf{h}})$ such that $\theta_{\rho,\chi_0}$ is of level $\Gamma[p,p]$. Then we repeat the same argument as in the proof of Theorem \ref{Lconv} but now we employ Proposition \ref{dirichletconv_2} instead of Proposition \ref{dirichletconv}, which is possible thanks to Proposition \ref{cuspidality_product} proved above.

\end{proof}

\section{Miscellaneous Loose Ends}

This final section spells out some of the limitations of the results of this paper, what could be done to circumvent these limitations, and possible avenues for further research following on from these results.

The need for the theta series to be vector-valued placed some limitations on which representations we could consider. Indeed  $\rho \otimes \det \in \tau(\Sigma)$ only when $\rho$ is trivial, and so the theta series is always of weight $\rho \otimes \det^{\frac{n}{2}}$. In contrast, the scalar case (i.e. $\rho$ trivial) allows a choice of $\mu\in\{0, 1\}$ and theta series of scalar weight $\frac{n}{2}+\mu$. Crucially this (the scalar case) meant the character $\chi$ could have arbitrary parity, since one just chooses $\mu$ so that $(\psi\chi)_{\infty}(-1) = (-1)^{[k]+\mu}$. In the present case, the assumption that $\chi$ is a character satisfying $(\psi\chi)_{\infty}(-1) = (-1)^{[k]}$ is needed, as seen in Theorems \ref{integralexp} and \ref{analyticity}, thus limiting the parity of the character $\chi$.

Theorem \ref{Lconv} is an attempt to extend the result that $L(s, f, \chi)\neq 0$ for all $\Re(s)>\frac{3n}{2}+1$ in the scalar case, see \cite{Sh00}, to the present case. A critical step of the proof requires $\theta$ to be a cusp form, which is easily achieved in the scalar case by taking $\mu = 1$. So a further ramification of the restriction on representations caused problems with this method, resulting in a weaker version (Theorem \ref{Lconv}) of the desired result. The desired result could be proven if we can simultaneously take $\chi$ to have arbitrary conductor and is such that $\theta_{\rho, \chi}$ is a cusp form for any choice of $\tau\in S_+$. Whilst we are able to take $\chi_0$ of Theorem \ref{Lconv} to have arbitrary conductor $p\neq 2$, the choice of $\tau\in S_+$ needed for Theorem \ref{theta} and the subsequent assumption that $f$ be $\chi_0$-ordinary for any such arbitrary choice of $\chi_0$ means this route is not viable. A stronger result on the cuspidality of the theta series is therefore needed for this method to be successful here.

We believe that our non-vanishing result can be used to obtain establish the algebraicity, after dividing by a suitable period, of some special values of the standard $L$-function studied here. Indeed, in the scalar weight case, the non-vanishing of the $L$-function is crucially used to obtain algebraicity results beyond the range of absolute convergence (see for example \cite{Sh00}), and seems plausible that the non-vanishing established in this paper can also be used to obtain similar results in the vector valued case. We hope to return to this in a future work.



\begin{thebibliography}{15}
\bibitem{AK} A.N. Andrianov, V.L. Kalinin, On the analytic properties of standard zeta functions of Siegel modular forms, Mat. Sb. 106 (148) (1978) 323-339; English
transl. Math. USSR Sb. 35 (1979), 1-17.
%
\bibitem{B} S. B\"{o}cherer, \"{U}ber die Funktionalgleichung automorpher L-Funktionen zur Siegel-schen Modulgruppe. J. R. Ang. Math. 362, 146 168 (1985)
%
\bibitem{BP} S. B\"{o}cherer, R. Schulze-Pillot, Siegel modular forms and theta series attached to Quaternion algebras, Nagoya Math. J. vol 121 (1991), 35-96
%
\bibitem{BP2} S. B\"{o}cherer, R. Schulze-Pillot, Siegel modular forms and theta series attached to Quaternion algebras II, Nagoya Math. J. vol 147 (1997), 71-106
%
\bibitem{BS} S. B\"{o}cherer, C-G. Schmidt, $p$-adic measures attached to Siegel modular forms, Annales de l'institut Fourier, (50) no 5, 2000
%
\bibitem{F} E. Freitag, Singular Modular Forms and Theta Relations, Lecture Notes in Mathematics, 1487, Springer 1991.
%
\bibitem{G57} R. Godement, S{\' e}minaire Cartan 10 (1957/58), Exp 4-9
%
\bibitem{KV} M. Kashiwara, M Vergne , On the Segal-Shale-Weil Representations and Harmonic Polynomials, Inventiones math. 44, 1-47 (1978)
%
\bibitem{Maass} H. Maass, Siegel's Modular Forms and Dirichlet Series, Lecture Notes in Mathematics, 216, Springer 1971.
%
\bibitem{PR1} Piatetski-Shapiro, Rallis, S.: L-functions for classical groups. Lecture notes, Institute for Advanced Study, 1984
%
\bibitem{PR2} Piatetski-Shapiro, Rallis, S.: A new way to get Euler products. J. R. Ang. Math. 392, 110-124 (1988)
%
\bibitem{Sh94} G.Shimura, Euler Products and Fourier coeffcients of automorphic forms of symplectic groups, Invent. math. 116, 531-576 (1994)
%
\bibitem{Sh95} G.Shimura, Eisenstein series and zeta functions on symplectic groups, Invent. math. 119, 539-584 (1995)
%
\bibitem{Sh97} G. Shimura, Euler Products and Eisenstein Series, CBMS, AMS, vol 93, 1997.
%
\bibitem{Sh00} G. Shimura, Arithmeticity in the Theory of Automorphic Forms, Mathematical Surveys and Monographs, AMS, vol 82, 2000.
%
\bibitem{W} R. Weissauer, Vektorwertige Siegelsche Modulformen kleinen Gewichtes, J. Reine Angew. Math., 343, 184-202 (1983)
%
\end{thebibliography}
\end{document}